%% file: nonlin_struct_drop_arXiv.tex
\documentclass[%
]{mpi2015-cscpreprint}
\usepackage{graphicx}
\usepackage{amssymb}
\usepackage{amsmath}
\usepackage{amsthm}

\usepackage{hyperref}
\hypersetup{
	colorlinks,
	linkcolor={red!50!black},
	citecolor={green!50!black},
	urlcolor={blue!80!black}
}
\usepackage{cleveref}
\usepackage{mathtools}
\usepackage{tikz}
\usepackage{pgfplots}
\usepackage{nicefrac}
\usepackage{todonotes}

\usepackage{pgfplots}

\newenvironment{customlegend}[1][]{%
	\begingroup
	\csname pgfplots@init@cleared@structures\endcsname
	\pgfplotsset{#1}%
}{%
	\csname pgfplots@createlegend\endcsname
	\endgroup
}%

\newlength\figureheight
\newlength\figurewidth
\newlength\fheight
\newlength\fwidth

\def\addlegendimage{\csname pgfplots@addlegendimage\endcsname}

\usepackage{xfrac}
\usepackage[software,hardware]{mymacros}
\usepackage{algorithm} 
\usepackage{algpseudocode}
\usetikzlibrary{matrix}
\usepackage{epstopdf}
\usepackage{mathrsfs}
\usepackage{rotating}

\newcommand{\bbm}{\begin{bmatrix}}
	\newcommand{\ebm}{\end{bmatrix}}
\usepgfplotslibrary{groupplots}
\usepackage{pgfplots}
\usepackage{pgfplotstable}
\usepackage{mathrsfs}
\usepackage{caption}
\usepackage{subcaption}
\usepackage{url}
\usepackage{enumitem}
\usepackage{soul}

\newcommand{\bhx}{\mathbf{\hat{x}}}
\newcommand{\bhy}{\mathbf{\hat{y}}}
\newcommand{\bhH}{\mathbf{\hat{H}}}
\newcommand{\bhA}{\mathbf{\hat{A}}}
\newcommand{\bhB}{\mathbf{\hat{B}}}
\newcommand{\bhC}{\mathbf{\hat{C}}}
\newcommand{\bhN}{\mathbf{\hat{N}}}

\newcommand{\chK}{\mathcal \hK}
\newcommand{\chL}{\mathcal \hL}

\newcommand{\chP}{\mathcal \hP}

\newcommand{\bten}[1]{\boldsymbol{\mathcal{#1}}}

\makeatletter
\newcommand{\pushright}[1]{\ifmeasuring@#1\else\omit\hfill$\displaystyle#1$\fi\ignorespaces}
\newcommand{\pushleft}[1]{\ifmeasuring@#1\else\omit$\displaystyle#1$\hfill\fi\ignorespaces}
\makeatother

\newcommand*\circled[1]{\tikz[baseline=(char.base)]{
		\node[shape=circle,draw,inner sep=0.2pt] (char) {#1};}
}
\newcommand\kronF[2]{#1^{\circled{\tiny{#2}}}}

\newcommand{\bbe}{\ensuremath{\mathbf{e}}}

\newcommand{\mphi}{\mathrm{\Phi}}

\renewcommand{\hat}[1]{\widehat{#1}}
\renewcommand{\tilde}[1]{\widetilde{#1}}

\newtheorem{theorem}{Theorem}

\newtheorem{definition}{Definition}
\newtheorem{lemma}{Lemma}
\newtheorem{remark}{Remark}

\makeatletter
\DeclareOldFontCommand{\rm}{\normalfont\rmfamily}{\mathrm}
\DeclareOldFontCommand{\sf}{\normalfont\sffamily}{\mathsf}
\DeclareOldFontCommand{\tt}{\normalfont\ttfamily}{\mathtt}
\DeclareOldFontCommand{\bf}{\normalfont\bfseries}{\mathbf}
\DeclareOldFontCommand{\it}{\normalfont\itshape}{\mathit}
\DeclareOldFontCommand{\sl}{\normalfont\slshape}{\@nomath\sl}
\DeclareOldFontCommand{\sc}{\normalfont\scshape}{\@nomath\sc}
\makeatother

\begin{document}

\title{
Dominant Subspaces of High-Fidelity Nonlinear Structured Parametric Dynamical Systems and Model Reduction
}


\author[$\ddagger$]{Pawan Goyal}
\affil[$\ddagger$]{Max Planck Institute of Dynamics of Complex Technical Systems\authorcr
	\email{goyalp@mpi-magdeburg.mpg.de}, \orcid{0000-0003-3072-7780}}

\author[$\dagger$]{Igor Pontes Duff}
\affil[$\dagger$]{Max Planck Institute of Dynamics of Complex Technical Systems\authorcr
	\email{pontes@mpi-magdeburg.mpg.de}, \orcid{0000-0001-6433-6142}}

\author[$\dagger\dagger$]{Peter Benner}
\affil[$\dagger\dagger$]{Max Planck Institute of Dynamics of Complex Technical Systems, Otto-von-Guericke University Magdeburg \authorcr
	\email{benner@mpi-magdeburg.mpg.de}, \orcid{0000-0003-3362-4103}}

\abstract{
In this work, we investigate a model order reduction scheme for high-fidelity nonlinear structured parametric dynamical systems. More specifically, we consider a class of nonlinear dynamical systems whose nonlinear terms are polynomial functions, and the linear part corresponds to a linear structured  model, such as second-order, time-delay, or fractional-order systems. Our approach relies on the Volterra series representation of these dynamical systems. Using this representation, we identify the kernels and, thus, the generalized multivariate transfer functions associated with these systems. Consequently, we present results allowing the construction of reduced-order models whose generalized transfer functions interpolate these of the original system at pre-defined frequency points.  For efficient calculations, we also need the concept of a symmetric Kronecker product representation of a tensor and derive particular properties of them. Moreover, we propose an algorithm that extracts dominant subspaces from the prescribed interpolation conditions. This allows the construction of reduced-order models that preserve the structure. We also extend these results to parametric systems and a special case (delay in input/output).  We demonstrate the efficiency of the proposed method by means of various numerical benchmarks.
}

\novelty{\begin{itemize}
	\item 
\end{itemize}}
\maketitle

\input{intro}

\input{probsetting}

\input{interpolationMOR}
\input{paramMOR}
\input{dropalgo}

\input{numerics_new}

\input{conclusions}

\section*{acknowledgements}
	We would like to express our gratitude to Dr.\ Steffen~W.~R.~Werner for providing the data and code from \cite{morBenGW20,morBenGW2020par}.	

\bibliographystyle{siamplain}
\bibliography{mor,igorBiblio}

\appendix

\section{Proof of Lemma 2.1}\label{proof:lemma21}
\begin{proof}
	\begin{itemize}
		\item[a).] First note that 
		\begin{equation}
			\kronF{\bx}{N} = \sum_{i_1 = 1}^n \cdots \sum_{i_{\mathrm N} = 1}^n \left(x_{i_1} x_{i_2} \cdots x_{i_n} \right) \left( \bbe_{i_1} \otimes \bbe_{i_2} \otimes \cdots\otimes \bbe_{i_n} \right).
		\end{equation}
		Since $x_{i_1} x_{i_2} \cdots x_{i_n} = x_{j_1} x_{j_2} \cdots x_{j_n} $, for every $(j_1,\ldots,j_n) \in \cS_\bi$, we can write 
		\begin{equation*}
			\kronF{\bx}{N} = \sum_{i_1 = 1}^n \cdots \sum_{i_{\mathrm N} = 1}^{n} \left(x_{i_1} x_{i_2} \cdots x_{i_n} \right) \left(\dfrac{1}{\alpha_\bi} \sum_{(j_1,\ldots,j_n) \atop\in \cS_\bi}\left( \bbe_{j_1} \otimes \bbe_{j_2} \otimes \cdots\otimes \bbe_{j_n} \right)\right).
		\end{equation*}
		Next, we have
		\begin{align*}
			\bH_{(1)} \kronF{\bx}{N} &=
			\sum_{i_1 = 1}^n  \cdots \sum_{i_{\mathrm N} = 1}^{n} \left(x_{i_1} x_{i_2} \cdots x_{i_n} \right) \bH_{(1)} \left(\dfrac{1}{\alpha_\bi} \sum_{(j_1,\ldots,j_n) \atop\in \cS_\bi}\left( \bbe_{j_1} \otimes \bbe_{j_2} \otimes \cdots\otimes \bbe_{j_n} \right)\right)\\
			&=
			\sum_{i_1 = 1}^n  \cdots \sum_{i_{\mathrm N} = 1}^{n} \left(x_{i_1} x_{i_2} \cdots x_{i_n} \right)  \left(\tilde\bH_{(1)} \left(:,i_1 + \sum_{l=2}^{\mathrm N}(i_l-1)\left(n^{l-1}\right) \right)  \right)\\
			&=
			\sum_{i_1 = 1}^n  \cdots \sum_{i_{\mathrm N} = 1}^{n} \left(x_{i_1} x_{i_2} \cdots x_{i_n} \right) \tilde\bH_{(1)} \left( \bbe_{i_1} \otimes \bbe_{i_2} \otimes \cdots\otimes \bbe_{i_n} \right)\\
			&= \tilde \bH_{(1)} \kronF{\bx}{N},
		\end{align*}
		which proves  part (a).
		\item [b).] We begin with
		\begin{align}
			\tilde\bH^{(1)} \left(\bq^{(1)}\otimes \cdots \otimes \bq^{(\mathrm N)}\right) & \nonumber\\
			& \hspace{-3cm}= \sum_{i_1=1}^n \cdots \sum_{i_{\mathrm N}=1}^n\tilde \bH_{(1)}\left(:, \left(i_1 + \sum_{l=2}^{\mathrm N}(i_l-1)n^l )\right)\right)  \left(q^{(1)}_{i_1} \cdots q^{(\mathrm N)}_{i_{\mathrm N}}\right) \nonumber\\
			& \hspace{-3cm} = \sum_{i_1=1}^n \cdots \sum_{i_{\mathrm N}=1}^n\bH_{(1)} \left(\sum_{(j_1,\ldots,j_n) \atop\in \cS_\bi}\dfrac{1}{\alpha_\bi}\left( \bbe_{j_1} \otimes \bbe_{j_2} \otimes \cdots\otimes \bbe_{j_n} \right) \right) \nonumber\\ 
			& \pushright{\left(q^{(1)}_{i_1} \cdots q^{(\mathrm N)}_{i_{\mathrm N}}\right).} \label{eq:tensor_per_vectors}
		\end{align}
		Since $(j_1,\ldots,j_n) \in \cS_\bi$, the above equation is invariant to the interchange of the indices $i_k$. Therefore, the Kronecker product of the $q_i$'s can appear in any order that would yield the same $\tilde \bH_{(1)} \left(\tilde\bq_1 \otimes \cdots \otimes \tilde\bq_{\mathrm N}\right)$, where  $\left(\tilde\bq_1, \ldots, \tilde \bq_{\mathrm N}\right)$ belongs to the set of all permutations of the set $\{\bq_1,\ldots,\bq_{\mathrm N}\}$. This proves the result.
		\item[c).] Assuming $l_j \in \{1,\ldots,n\}$ for $j \in\{1,\ldots, \mathrm N+1\}$, we have
		\begin{align*}
			&\bbe_{l_2}\tilde\bH_{(2)} \left(\bbe_{l_{\mathrm N +1}} \otimes \bbe_{l_{\mathrm N}} \otimes \cdots \otimes \bbe_{l_{3}} \otimes \bbe_{l_{1}}  \right) \\
			&\hspace{1.2cm}\text{using \eqref{eq:matricizationRelation}}\\
			&\hspace{1cm}= \bbe_{l_1}\tilde\bH_{(1)} \left(\bbe_{l_{\mathrm N +1}} \otimes \bbe_{l_{\mathrm N}} \otimes \cdots \otimes \bbe_{l_{3}} \otimes \bbe_{l_{2}}  \right)  \\
			& \hspace{1.2cm}\text{using \eqref{eq:tensor_per_vectors}}\\
			&\hspace{1cm}= \bbe_{l_1}\tilde\bH_{(1)} \left(\bbe_{l_{\mathrm N +1}} \otimes \bbe_{l_{\mathrm N}} \otimes \cdots  \otimes \bbe_{l_{m+1}} \otimes \bbe_{l_2} \otimes \bbe_{l_m} \otimes \bbe_{l_{m-1}} \otimes \bbe_{l_{3}}\right)\\
			& \hspace{1.2cm} \text{using \eqref{eq:matricizationRelation}}\\
			&\hspace{1cm}= \bbe_{l_2}\tilde\bH_{(m)} \left(\bbe_{l_{\mathrm N +1}} \otimes \bbe_{l_{\mathrm N}} \otimes \cdots  \otimes \bbe_{l_{m+1}} \otimes \bbe_{l_m} \otimes \bbe_{l_{m-1}} \otimes \bbe_{l_{3}} \otimes \bbe_{l_1}\right).
		\end{align*}
		This shows that the entries in $\tilde\bH_{(2)}$ and $\tilde\bH_{(m)}$ ($m\geq 2$) are equal, implying  the result.
	\end{itemize}
\end{proof}

\input{fundsol}
\input{tangetialinterp}
\end{document}

%% file: intro.tex

\section{Introduction}
Dynamical systems are the basic framework used for modeling, controlling, and analyzing a large variety of engineering processes. Due to the increasing use of dedicated computer-based modeling design software, numerical simulation is now used more frequently to understand the dynamics of a complex system and to shorten both development time and cost. However, the need for enhanced model accuracy inevitably leads to an increasing number of variables and resources, which entails a high computational cost. In this context, model order reduction (MOR) is a possible remedy for such complex simulations. Precisely, MOR aims to replace a complex high-fidelity model with a reduced-order model (ROM) that mimics a certain dynamical behavior of the original model and preserves its features. As a result, this alleviates the numerical burden and reduces the computational time.

Many MOR methods for (parametric) nonlinear systems are based on simulated data. This means that, for given inputs and parameters, \emph{snapshots} of the state vector $\bx$ are collected \cite{benner2020model}. Then, a low-dimensional dominant subspace is  determined by means of a singular value decomposition (SVD) of the matrix containing the collected snapshots as columns. Hence, a ROM is constructed via Galerkin projection. Among these methods, proper orthogonal decomposition is arguably the most favored method (see, e.g.,~\cite{morGubV17} for more details). Additionally, for nonlinear systems, this approach is often combined with hyper-reduction methods, such as EIM~\cite{morBarMNetal04}, DEIM~\cite{morChaS10}, and GNAT \cite{carlberg2013gnat}, allowing fast evaluation of nonlinear terms.  Also, for parametric problems, reduced basis methods have been successfully applied to several nonlinear systems, see, e.g., \cite{morQuaMN15}. Although these methods have been successful in several applications, they are \emph{input-dependent}, i.e., the quality of the ROMs  depends on the choices of input functions and parameters used to collect the snapshots.  Hence, it may be harder to obtain a ROM independent of inputs, which are suitable, e.g., for control problems. 

In this work, we focus on MOR methods that are \emph{input-independent}, i.e., ROMs can approximate the high-fidelity model behavior for all admissible inputs. The reader is referred to \cite{morBauBF14} for an overview of input-independent MOR methods.  These methods, broadly speaking, are divided into two classes: interpolation-based and balanced truncation approaches. In this work, we focus on the class of interpolation-based methods. For the class of nonlinear systems, interpolation-based MOR methods have been extended to certain classes of nonlinear systems, e.g., bilinear (see, e.g., \cite{morAntGI16,morBreD10}), parametric bilinear (see, e.g., \cite{morRodGB18}), quadratic-bilinear (see, e.g., \cite{morAhmBJ16,morBenB15,morGosA18}) and, more recently, (parametric-)polynomial systems (see, e.g., \cite{morBenG19,morBenG21}).  

Besides the non-linearities, in many applications, dynamical systems possess a particular dynamical structure, e.g., second-order, time-delay, and fractional-order systems. We call these systems \emph{structured}.  There exist several MOR methods for structured systems, allowing to preserve such dynamical structures in a ROM. We refer to  \cite{morChaLVD06,morEidSLetal07,morReiS08} for second-order systems and \cite{jarlebring2013model} for time-delay systems.  Moreover, in \cite{morBeaG09}, the authors propose a framework allowing interpolation-based MOR for a vast class of linear  structured systems. This framework was extended to the class of parametric linear structured systems in \cite{morAntBG10},  and a data-driven identification approach was proposed in \cite{morSchU18}. Additionally, in \cite{morBenGP19}, the authors have proposed an approach to determine the dominant subspaces of a given (parametric) linear structured model.
It is worth noticing that a balanced truncation approach is proposed in \cite{breiten2016structure} for this class of systems.

In this paper, we focus on MOR for structured systems with polynomial non-linearities.  To illustrate this class of systems, let us consider the  bilinear time-delay system presented in \cite{morGosPBetal19}, which is of the form:
\begin{subequations}\label{eq:bildelay}
	\begin{align}
		\dot{\bx}(t) - \bA_1 \bx(t) -  \bA_{2}\bx(t-\tau) &=  \bN \bx(t)\bu(t) +\bB \bu(t), \label{eq:bildelay_control}\\
		\by(t) &= \bC \bx(t), \label{eq:bildelay_obs}
	\end{align}
\end{subequations}
where $\bx(t) \in \R^n$, $\bu(t) \in \R^m$ and $\bu(t)\in \R^p$ are the state, input, and output vectors.
The left-hand side of \eqref{eq:bildelay_control} corresponds  to the linear part of the dynamics, governed by a time-delay structure. On the right-hand side, the term $\bN \bx(t) \bu(t)$ corresponds to a bilinear non-linearity. Our main goal is to determine a ROM for the system \eqref{eq:bildelay} that preserves its structure as well, i.e., to  search for a  surrogate model of the form:
\begin{subequations}\label{eq:ROMbildelay}
	\begin{align}
		\dot{\bhx}(t) - \bhA_1 \bhx(t) -  \bhA_{2}\bhx(t-\tau) &=  \bhN \bhx(t)\bu(t) +\bhB \bu(t), \label{eq:ROMbildelay_control}\\
		\bhy(t) &= \bhC \bhx(t), \label{eq:ROMbildelay_obs}
	\end{align}
\end{subequations}
where $\bhx(t) \in \R^r$ is the reduced state vector and $r\ll n$. This ROM should have a similar input-output behavior as the original one for all admissible inputs. Note that interpolation-based MOR for (parametric-) bilinear structured  systems has been  proposed in \cite{morBenGW2020par,morBenGW20}.
In this paper, we build upon the methodologies proposed in \cite{Wer21} for structured bilinear and quadratic-bilinear systems, and in \cite{morBenG21} for polynomial systems. 
Particularly, we focus on an interpolation-based MOR approach for structured systems with polynomial non-linearities. Additionally, by following the philosophy in \cite{morBenG19,morBenG21,morBenGP19}, we propose an algorithm enabling to determine the dominant subspace information from interpolation points.

The remaining structure of the paper is as follows. In \Cref{sec:setup}, we present the class of polynomial structured systems, and provide some results for symmetric tensor representations that play an important role in MOR.  In \Cref{sec:nonpar_interpolation}, we discuss the Volterra series representation of polynomial structured systems. This representation allows us to identify the system kernels, thus enabling us to define the generalized transfer functions of the system. Based on these, we present results that yield ROMs whose generalized transfer functions interpolate those of the original model at pre-defined interpolation points. Then, in \Cref{sec:par_interpolation}, these results are generalized to other classes, such as parametric and input-output delay systems. Moreover, in \Cref{sec:DROP}, we propose an algorithm to determine the dominant subspace information from a given large set of interpolation points. In \Cref{sec:Numerics}, we illustrate the efficiency of the proposed algorithm by means of three benchmark problems and compare it with the state-of-the-art. We conclude the paper with a summary of our contributions and future perspectives. 

We make use of the following notation in the paper:
\begin{itemize}
	\item Matrices and vectors are denoted with bold symbols, e.g., $\bA, \bB, \bv, \bw$.
	\item The $i$th entry of a vector $\bq\in \Rn$ is denoted by $\bq_i$.
	\item The Kronecker product is denoted by `$\otimes$'.
	\item $\bI_m$ is the identity matrix of size $m\times m$.
	\item $\bbe_i$ denotes the $i$th column of the identity matrix of appropriate size. 
	\item $\cV^{\circled{\tiny{$\xi$}}}$ is a short-hand notation for $\underbrace{\cV\otimes \cdots\otimes \cV}_{\xi-\text{times}}$, where $\cV$ is a vector/matrix.
\end{itemize}

%% file: probsetting.tex
\section{Problem Setting and Tensor Algebra}\label{sec:setup}
\subsection{Problem setting}
In this paper, we focus on  structured dynamical systems with polynomial non-linearity.  These systems can be written in the form: 
\begin{subequations}\label{eq:struct_nonlin_sys}
	\begin{align}
		\left(\cL \bx\right)(t)	&= \cP( \bx(t), \bu(t)) +\bB \bu(t), \label{eq:struct_nonlin_sys_dyn} 
		\\ 
		\by(t) 		  &= \bC \bx(t), \label{eq:struct_nonlin_sys_obs} 
	\end{align}
\end{subequations}
with matrices $ \bB \in \Rnm$, $\bC \in \Rpn$; the state, input, and output vectors are denoted by $ \bx(t) \in \Rn$, $\bu(t) \in \Rm$ and $\by(t) \in \Rp$, respectively; $\cL(\cdot)$ corresponds to a linear operator, while $\cP(\cdot): \R^{n+m} \rightarrow \Rn$ represents non-linear terms. Additionally, we assume the corresponding initial conditions for \eqref{eq:struct_nonlin_sys_dyn} to be zero.  We shall discuss its variants (e.g., the parametric version) in \Cref{sec:par_interpolation}.

We consider linear operators $\cL(\cdot)$ in the system \eqref{eq:struct_nonlin_sys}, covering a large class of systems arising in various science and engineering applications, e.g., classical linear systems, second-order systems, time-delay systems, and integro-differential systems. We list some examples in \Cref{tab:LinearStructureExamples}.  Furthermore, we assume that the non-linear function $\cP(\cdot)$ corresponds to polynomial non-linearities as follows: 
%
%
\begin{equation}\label{eq:PolynomialOpt}
	\cP(\bx(t), \bu(t)) = \sum_{\mathrm \xi = 2}^d \bH_{\xi}\kronF{\bx}{$\xi$}(t) + \sum_{\eta = 1}^{d-1} \bN_{\eta}\left(\bu(t)\otimes \bx^{\circled{\tiny {$\eta$}}}(t)\right),
\end{equation}
where $\bH_\xi \in \R^{n\times n^\xi}$, $\xi \in \{2,\ldots,d\},\bN_\eta \in \R^{n\times m\cdot n^\eta}$, $\eta \in \{1,\ldots,d-1\}$. In this work, the system~\eqref{eq:struct_nonlin_sys} is refereed to as \emph{polynomial structured system}. 

\begin{table}[tb]
	\centering
	\caption{Examples of common linear operators in dynamical systems.}
	\label{tab:LinearStructureExamples}
	{\footnotesize \renewcommand{\arraystretch}{2} 
		\begin{tabular}{|l|c|c|} \hline
			& \textbf{Linear operator $\cL\bx(t)$ } & \textbf{Frequency-domain description $\cK(s)\bX(s)$}\\ 
			\hline
			first-order  &  $ \bE\dot{\bx}- \bA\bx(t)   $ &  $\left(s\bE - \bA\right)\bX(s)$  
			\\ \hline
			second-order & $\bM\ddot{\bx}(t) + \bD\dot{\bx}(t)  + \bK\bx(t)$ & $\left(s^2\bM + s\bD  + \bK\right)\bX(s)$
			\\ \hline
			state delay & $\bE\dot{\bx}(t) - \bA\bx(t) - \bA_{\tau}\bx(t-\tau)$ & $\left(s\bE - \bA - \mathrm{e}^{-\tau s}\bA_{\tau}\right)\bX(s)$ 
			\\\hline
			fractional order &$\bE\dot{\bx}(t) - \displaystyle\frac{1}{\mathrm\Gamma(\alpha)}\displaystyle\int_0^t s^{\alpha-1}\bA x(t-s)ds  $     & $\left(s\bE - s^{-\alpha}\bA\right)\bX(s)$ \\ \hline
		\end{tabular}
	}
\end{table}

In this paper, our aim is to construct ROMs of order $r$ that have the same structure as in~\eqref{eq:struct_nonlin_sys}:
\begin{subequations}\label{eq:ROMstruct_nonlin_sys}
	\begin{align}
		\left(\chL \hat\bx\right)(t)	&= \chP(\hat\bx(t), \bu(t)) +\bhB \bu(t), \label{eq:ROMstruct_nonlin_sys_dyn} 
		\\ 
		\hy(t) 		  &= \bhC \hx(t), \label{eq:ROMstruct_nonlin_sys_obs} 
	\end{align}
\end{subequations}
where $ \bhB \in \R^{r \times m}$, $\bhC \in \R^{p \times r}$; the reduced state, input, and approximated output vectors are denoted by $ \hat\bx(t) \in \Rr$, $\bu(t) \in \Rm$ and $\hat\by(t) \in \Rp$, respectively, with $r\ll n$, and $\hat\by(t)$ approximates very well $\by(t)$ for all admissible inputs.   Additionally, the reduced linear operator $\chL$ has also the same structure as the linear operator $\cL$. 
As an example, for a second order dynamical system, the original linear operator is represented by
\begin{equation*}
	\cL(\bx) = \bM \ddot{\bx}(t) + \bD\dot{\bx}(t) + \bK\bx(t),
\end{equation*}
where  $\bM, \bD, \bK \in \R^{n \times n}$. Hence, in order to preserve the system structure, the reduced linear operator has to possess the following form: 
\begin{equation*}
	\chL(\hat\bx) = \hat\bM \ddot{\hat\bx}(t) + \hat\bD\dot{\hat\bx}(t) + \hat\bK\hat\bx(t),
\end{equation*}
with $ \hat\bM, , \hat\bD, \hat\bK \in \R^{r\times r}$.
Also, the reduced non-linear function $\chP$ has the same structure as the original non-linear term $\cP$, i.e.,
%
\begin{equation}\label{eq:ROMPolynomialOpt}
	\chP(\hat\bx(t)) = \sum_{\xi = 2}^d \bhH_{\xi}\kronF{\hat\bx}{$\xi$}(t) + \sum_{\eta = 1}^{d-1} \bhN_{\eta}\left(\bu(t)\otimes \hat\bx^{\circled{\tiny {$\eta$}}}(t)\right).
\end{equation}
We aim at achieving this goal via a Petrov-Galerkin projection. This means, we require two matrices $\bV, \bW \in \Rnr$ such that the reduced operator $\chL(\cdot)$ in \eqref{eq:ROMstruct_nonlin_sys} and matrices involved in determining the polynomial term in~\eqref{eq:ROMPolynomialOpt} can be given as follows:
%
%
%
\begin{equation}\label{eq:RedMatrices}
	\begin{aligned}
		\chL(\cdot) &= \bW^{\top} \cL(\cdot) \bV, \quad 	\bhH_{\xi} = \bW^{\top}\bH_{\xi}\bV^{\circled{\tiny{$\xi$}}},\quad \xi \in \{2,\ldots,d\}, \\    \bhN_{\eta} &= \bW^{\top}\bN_{\eta}\bV^{\circled{\tiny{$\eta$}}},\quad \eta\in \{1,\ldots,d-1\}.
	\end{aligned}
\end{equation}
Clearly,  the selection of the projection matrices $\bV$ and $\bW$ plays an important role in determining the desired ROMs. In this paper, we aim to determine these matrices such that the resulting ROM fulfills certain interpolation properties. We mention that for structured bilinear cases ($d = 1$), some interpolation-based results were developed in \cite{morBenGW20}, which we generalize to the polynomial case in the next section.

\subsection{Results on Tensor Algebra}\label{sec:tensoralg}
In this subsection, we recall some tensor algebra concepts that will be useful later in the paper and discuss symmetric tensors. A motivation for that is that the matrices $\bH_\xi $ and $\bN_\eta$ appearing in  \eqref{eq:PolynomialOpt} can be interpreted as unfoldings of  higher-order tensors. The process of representing a tensor as a matrix is often referred to as  \emph{matricization}, see, e.g., \cite{KolB09,kolda2006multilinear}. Typically, an $\mathrm N$th order tensor can be unfolded in $\mathrm N$ different ways, depending on the dimension  along which the tensor is unfolded. So, we begin by recalling the definition of matricization. 
\begin{definition}[e.g., \cite{kolda2006multilinear}]
	Consider an $\mathrm N$th order tensor $\bten{X} \in  \R^{n_1\times \cdots\times n_{\mathrm N}}$. The mode-$n$ matricization of the tensor $\bten{X}$, denoted by $\boldsymbol X_{(m)}$, is obtained by the following mapping:
	\begin{equation*}
		\boldsymbol{X}_{(m)}(i_m,j) = \bten{X}(i_1,\ldots,i_{\mathrm N}),
	\end{equation*}
	where $j = 1+ \sum\limits_{k=1,k\neq m}^{\mathrm N}(i_k-1)J_k$ with $J_k = \prod\limits_{z=1,z\neq m}^{k-1}n_z$ and $i_m \in \{1, \dots, n_m\}$.
\end{definition}

Next, we recall a connection between the mode-$n$ matricization and Kronecker products from, e.g., \cite{kolda2006multilinear}. For this, let the tensor-matrix product be:
\begin{equation*}
	\bten{Y} = \bten{X} \times_1 \mathrm \bA^{(1)} \times_2 \mathrm \bA^{(2)} \cdots \times_N\mathrm \bA^{(\mathrm N)},
\end{equation*}
where $\mathrm \bA^{(l)} \in \R^{J_l\times n_l }$, $\bten{Y}  \in \R^{J_1 \times \dots \times J_N}$, and $\times_i$ denotes the tensor contraction with respect to the $i$th dimension of the tensor. Then, the following relation between the unfolded tensors and Kronecker products holds (\cite[Prop. 3.7]{kolda2006multilinear}):
\begin{equation}\label{eq:matricizationRelationMatrix}
	\bY_{(m)} =\mathrm \bA^{(m)}\bX_{(m)}\left(\mathrm \bA^{(\mathrm N)}\otimes \cdots\otimes\mathrm \bA^{(m+1)} \otimes\mathrm \bA^{(m-1)} \otimes \mathrm \bA^{(1)} \right)^\top, \quad m\in \{1, \dots, \mathrm N\}.
\end{equation}
Now, we discuss a special case---that is, if $\bA^{(l)} = \ba_l^\top$, where $\ba_l$ is a column vector and $l \in \{1,\ldots,\mathrm N\}$. In this case, we observe that $\bY_{(m)}$ becomes a scalar,  given as:
\begin{equation}\label{eq:matricizationRelation}
	\boldsymbol{\bY}_{(m)} = \ba_m^\top\boldsymbol{\bX}_{(m)}\left( \mathrm \ba_\mathrm N\otimes\cdots\otimes \mathrm \ba_{m+1}\otimes \mathrm  \ba_{m-1}\otimes \cdots \otimes \mathrm \ba_1 \right), \quad m \in \{1,\ldots,\mathrm N\}.
\end{equation}
In what follows, we provide a result in tensor calculus  that is of particular interest for the setting considered in this paper. 

Before we proceed further, we define some notations. Let $\cS$ be a set $\{i_1,i_2,\ldots,i_n\}$, and  denote the set of all permutations of $\cS$ by $\cS_\bi$ and the number of elements in $\cS_\bi$ by $\alpha_\bi$. For example, consider a set $\{1,2,3\}$. Then, its permutations are: $(1,2,3), (2,1,3), (1,3,2), (2,3,1), (3,1,2)$ and $(3,2,1)$, and the number of elements are six. Having said this, in the following, we provide a result on the symmetric representation of the tensors.
\begin{lemma}\label{lemma:sym_tensor}
	Consider an $(\mathrm N+1)$st order tensor $\cH \in \R^{n\times \cdots \times n}$.  Let us consider a set $\cS \in \{i_1,\ldots,i_n\}$ and denote the set containing all its permutations by $\cS_\bi$ and the number of elements in $\cS_\bi$ by $\alpha_\bi$. Furthermore, let a tensor $\tilde \cH$ be defined such that the $\omega := \left(i_1+\sum_{l=2}^{\mathrm N} (i_l-1)(n^{l-1})\right)$th column of its mode-1 matricization (denoted by $\tilde \bH_{(1)}$) is given as follows:
	\begin{equation}
		\tilde \bH_{(1)}(:,\omega):=\bH_{(1)} \left(\sum_{(j_1,\ldots,j_n) \in \cS_\bi}\dfrac{1}{\alpha_\bi}\left( \bbe_{j_1} \otimes \bbe_{j_2} \otimes \cdots\otimes \bbe_{j_n} \right) \right) ,
	\end{equation}
	where $\bH_{(1)}$ is the mode-$1$ matricization of the tensor $\cH$.  
	Then, the following conditions are satisfied:
	\begin{enumerate}
		\item[a).] $\bH_{(1)}\left(\bx \otimes \cdots \otimes \bx\right) = \tilde\bH_{(1)}\left(\bx \otimes \cdots \otimes \bx\right)$, where  $\bx \in \Rn$ is a vector. 
		\item[b).] $\tilde\bH_{(1)}\left(\bq^{(1)}\otimes\cdots \otimes \bq^{(\mathrm N)}\right) = \tilde\bH_{(1)} \left(\tilde\bq^{(1)}\otimes\cdots\otimes \tilde\bq^{(\mathrm N)}\right)$, where $\bq^{(i)} \in \R$, $i \in \{1,\ldots,\mathrm N\}$, and $\left(\tilde\bq^{(1)},\ldots,\tilde\bq^{(\mathrm N)} \right)$ belongs to the set of all permutations of  $\left\{\bq^{(1)},\ldots, \bq^{(\mathrm N)}\right\}$.
		\item [c).] Moreover, all mode-m  matricizations of the tensor $\tilde\cH$  for $m \geq 2$  are the same, i.e., $\tilde\bH_{(2)}  =\tilde\bH_{(3)} = \cdots = \tilde\bH_{(\mathrm N)}$.
	\end{enumerate}
\end{lemma}
\begin{proof}
	The proof is given in \Cref{proof:lemma21}. 
\end{proof}
For a better understanding, we illustrate \Cref{lemma:sym_tensor} by an example. Consider a tensor $\cH^{2\times 2\times 2\times 2}$ such that its mode-1 matricization (denoted by $\bH_{(1)}$) is given by
\begin{equation}
	\bH_{(1)} = \begin{bmatrix}a_1 ~~& a_2 ~~& a_3 ~~& a_4 ~~& a_5 ~~& a_6 ~~& a_7 ~~& a_8 \\b_1 ~~& b_2 ~~& b_3 ~~& b_4 ~~& b_5 ~~& b_6 ~~& b_7 ~~& b_8  \end{bmatrix}.
\end{equation}
Next, we write down explicitly the term $\bH_{(1)} (\bx\otimes \bx\otimes \bx)$, where $\bx = \bbm x_1 & x_2 \ebm^\top$-- that is
\begin{equation*}
	\bbm a_1 x_1^3 + a_2 x_1^2 x_2 + a_3 x_1x_2x_1 + a_4 x_1x_2^2 + a_5 x_2x_1^2 + a_6 x_2x_1x_2 + a_7x_2^2x_1 + a_8x_2^3 \\
	b_1 x_1^3 + b_2 x_1^2 x_2 + b_3 x_1x_2x_1 + b_4 x_1x_2^2 + b_5 x_2x_1^2 + b_6 x_2x_1x_2 + b_7x_2^2x_1 + b_8x_2^3 
	\ebm. 
\end{equation*}
Now, we define 
\begin{align*}
	\ta_2 &= \tfrac{(a_2 + a_3 + a_5)}{3}, & \ta_4 &= \tfrac{(a_4 + a_6 + a_7)}{3} &
	\tb_2 &= \tfrac{(b_2 + b_3 + b_5)}{3} &\tb_4 &= \tfrac{(b_4 + b_6 + b_7)}{3}
\end{align*}
Then, we can also write $\bH_{(1)} (\bx\otimes \bx\otimes \bx)$ as:
\begin{equation*}
	\bbm a_1 x_1^3 + \ta_2 x_1^2 x_2 + \ta_2 x_1x_2x_1 + \ta_4 x_1x_2^2 + \ta_2 x_2x_1^2 + \ta_4 x_2x_1x_2 + \ta_4x_2^2x_1 + a_8x_2^3 \\
	b_1 x_1^3 + \tb_2 x_1^2 x_2 + \tb_2 x_1x_2x_1 + \tb_4 x_1x_2^2 + \tb_4 x_2x_1^2 + \tb_4 x_2x_1x_2 + \tb_4x_2^2x_1 + b_8x_2^3 
	\ebm. 
\end{equation*}
Consequently, if we define a tensor $\tilde \cH$ such that its mode-1 matricization is given as
\begin{equation}
	\tilde\bH_{(1)} = \begin{bmatrix}a_1 ~~& \ta_2 ~~& \ta_2 ~~& \ta_4 ~~& \ta_2 ~~& \ta_4 ~~& \ta_4 ~~& a_8 \\ b_1 ~~& \tb_2 ~~& \tb_2 ~~& \tb_4 ~~& \tb_2 ~~& \tb_4 ~~& \tb_4 ~~& b_8  \end{bmatrix},
\end{equation}
then $\tilde \bH_{(1)}(\bx\otimes \bx\otimes \bx) = \bH_{(1)}(\bx\otimes \bx\otimes \bx)$. Moreover, it can also be observed that 
$\tilde \bH_{(1)} \left(\bu\otimes \bv\otimes \bw\right) = \tilde \bH_{(1)} \left(\bv\otimes \bu\otimes \bw\right)  = \cdots =\tilde \bH_{(1)} \left(\bw\otimes \bv\otimes \bu\right)$. Furthermore, if one aims at obtaining $\tilde\bH_{(1)}$ from \Cref{lemma:sym_tensor}, then we can write, e.g., its 2nd column as 
\begin{align*}
	\tilde\bH_{(1)}(:,2) &= \bH_{(1)} \left(\dfrac{1}{3} \left(\left(\bbe_2\otimes \bbe_1\otimes \bbe_1\right) + \left(\bbe_1\otimes \bbe_2\otimes \bbe_1\right) + \left(\bbe_1\otimes \bbe_1\otimes \bbe_2\right)  \right)\right)\\
	& =\dfrac{1}{3} \bbm a_2 + a_3 + a_5 \\ b_2 + b_3 + b_5 \ebm.
\end{align*}

To sum up, according to  \Cref{lemma:sym_tensor}, a tensor $\cH$ can be symmetrized (denote the symmetrized tensor by $\tilde \cH$) without changing the quantity $\bH_{(1)}(\bx\otimes \cdots\otimes \bx)$ and the following commutation rule is also fulfilled:
\begin{equation}
	\tilde \bH_{(1)} (\bv_1\otimes \cdots\otimes \bv_n) = 	\tilde \bH_{(1)} (\tilde \bv_1\otimes \cdots \otimes\tilde \bv_n),
\end{equation}
for every permutation $(\tilde \bv_1,\ldots, \tilde \bv_n)$ of the set $\{\bv_1,\ldots, \bv_n\}$. This extends the discussion in \cite{morBenB15} for $3$rd order tensors to the general case. Therefore, in the rest of the paper, without loss of generality, we assume that all tensors associated with $\bH_{\xi}$ and $\bN_{\eta}$ are symmetric.

%% file: interpolationMOR.tex
\section{Volterra Series and Interpolation-Based MOR}\label{sec:nonpar_interpolation}
This section presents the Volterra series representation of nonlinear structured systems~\eqref{eq:struct_nonlin_sys}. For this, we extend  the discussion on non-structured polynomial systems in \cite{morBenG21} to the case of structured polynomial systems.  We aim at identifying the kernels related to the system~\eqref{eq:struct_nonlin_sys} that allow us to define generalized transfer functions. As a consequence, we intend to construct a ROM such that its generalized transfer functions interpolate these of the original system at pre-defined interpolation points. For simplicity, in this section, we assume  the system to be single-input single-output (SISO), i.e., $m = p =1$. We will also provide  an extension to the multi--input multi--output case~(MIMO).
\subsection{Volterra series representation}
First, let $\mphi$ be  the fundamental solution associated with the linear operator $\cL$. We detail the meaning of the fundamental solution in \Cref{appen:FunSol}. Consequently, the solution of an equation of the form:
\begin{equation}
	(\cL \bx)(t) = \bg(t)
\end{equation}
can be given as the convolution (using the fact that $\bx(0)=0$ by assumption)
\begin{equation}\label{eq:LinCovolution}
	\bx(t) = \int_0^{t} \mphi(\sigma)\bg(t-\sigma)d\sigma.
\end{equation}
Using this, we can write the solution of \eqref{eq:struct_nonlin_sys_dyn} as follows:
\begin{multline}\label{eq:system_state}
	\bx(t) = \int_0^t \mphi(\sigma_1) \bB \bu(t_{\sigma_1})d\sigma_1 + \sum_{\xi = 2}^d \int_0^t \mphi(\sigma_1) \bH_\xi \bx^{\circled{\tiny {$\xi$}}} (t_{\sigma_1}) d\sigma_1 \\ 
	+ \int_0^t \sum_{\eta = 1}^{d-1} \mphi(\sigma_1) \bN_{\eta} \left.\bx^{\circled{\tiny {$\eta$}}}(t_{\sigma_1})\right.  \bu(t_{\sigma_1}) d\sigma_1,
\end{multline}
where $t_{\sigma_1} := t-\sigma_1$. Moreover, we can determine the expression for $\bx(t_{\sigma_1})$ using the above equation -- that is,
\begin{multline}\label{eq:expressionxsigma1}
	\bx(t_{\sigma_1}) = \int_0^{t_{\sigma_1}} \mphi(\sigma_2) \bB \bu(t_{\sigma_{1}}{-}\sigma_2)d\sigma_2 + \sum_{\xi = 2}^d\int_0^{t_{\sigma_1}} \mphi(\sigma_2) \bH_\xi \bx^{\circled{\tiny {$\xi$}}}(t_{\sigma_1}{-}\sigma_2) d\sigma_2 \\
	+ \sum_{\eta = 1}^{d-1} \int_0^{t_{\sigma_1}} \mphi(\sigma_2) \bN_\eta\left. \bx^{\circled{\tiny {$\eta$}}}(t_{\sigma_1}{-}\sigma_2) \right. \bu(t_{\sigma_1}{-}\sigma_2) d\sigma_2.
\end{multline}
We utilize the expression for $\bx(t_{\sigma_1})$ in \eqref{eq:system_state}, multiplied  by the matrix $\bC$. Hence, we obtain
\begin{subequations}\allowdisplaybreaks
	\begin{align*}
		\by(t) &= \int_0^t \bC \mphi(\sigma_1) \bB \bu(t_{\sigma_1})d\sigma_1 
		\\
		&\qquad+ \sum_{\xi = 2}^d \int_0^t \underbrace{\int_0^{t_{\sigma_1}}{\cdots} \int_0^{t_{\sigma_1}}}_{\xi-\text{times}} \bC\mphi(\sigma_1)  \bH_\xi \left(\mphi(\sigma_2)\bB \otimes \cdots \otimes \mphi(\sigma_{\xi+1})\bB \right) 
		\\
		&\qquad \times \left(\bu(t_{\sigma_1}-\sigma_2) \cdots  \bu(t_{\sigma_1}-\sigma_{\xi+1}) \right) d\sigma_1 d\sigma_2\cdots d\sigma_{\xi+1}
		\\
		&\qquad + \sum_{\eta = 1}^{d-1} \int_0^t \underbrace{\int_0^{t_{\sigma_1}}\cdots \int_0^{t_{\sigma_1}}}_{\eta-\text{times}} \bC \mphi(\sigma_1) \bN_\eta\left(\mphi(\sigma_2)\bB \otimes \cdots \otimes \mphi(\sigma_{\eta+1})\bB \right)
		\\
		&\qquad \times \left(\bu(t_{\sigma_1})  \bu(t_{\sigma_1}-\sigma_2) \cdots  \bu(t_{\sigma_1}-\sigma_{\eta+1}) \right) d\sigma_1 d\sigma_2\cdots d\sigma_{\eta+1} + \cdots,
	\end{align*}
\end{subequations}
with the assumption that the above series convergences. As a result, we have an analytical representation of the output $\by(t)$ as an infinite sum of multivariable convolutions. But in the above expression for $\by(t)$, we only note down the leading few \emph{kernels} of the Volterra series for the system~\eqref{eq:struct_nonlin_sys}. In this paper, we focus only on those:
%
\begin{subequations}
	\begin{align}
		f_{\mathrm L}(\sigma_1) &:= \bC\mphi(\sigma_1) \bB,\\
		f^{(\xi)}_{\mathrm H}(\sigma_1,\ldots,\sigma_{\xi+1}) &:= \bC\mphi(\sigma_1)  \bH_\xi \left(\mphi(\sigma_2)\bB \otimes \cdots \otimes \mphi(\sigma_{\xi+1})\bB\right), \\
		f^{(\eta)}_{\mathrm N}(\sigma_1,\ldots,\sigma_{\eta+1}) &:= \bC \mphi(\sigma_1) \bN_\eta\left(\mphi(\sigma_2)\bB \otimes \cdots \otimes \mphi(\sigma_{\eta+1})\bB \right),
	\end{align}
\end{subequations}
with $\xi \in \{2,\ldots,d\}$ and $\eta \in \{1,\ldots,d-1\}$. Furthermore, by taking the multivariate Laplace transform (see, e.g., \cite{rugh1981nonlinear}), we get the frequency-domain representations of the kernels as follows:
\begin{subequations}\label{eq:general_TF}
	\begin{align}
		\bF_{\mathrm L}(s_1) &:=  \bC\cK^{-1}(s_1)\bB,\\
		\bF^{(\xi)}_{\mathrm H}(s_1,\ldots,s_{\xi+1}) & :=  \bC\cK^{-1}(s_{\xi+1}) \bH_\xi \left(\cK^{-1}(s_\xi)\bB \otimes \cdots \otimes \cK^{-1}(s_{1})\bB\right),\\
		\bF_{\mathrm N}^{(\eta)}(s_1,\ldots, s_{\eta+1}) &:=  \bC \cK^{-1}(s_{\eta+1}) \bN_\eta\left(\cK^{-1}(s_{\eta})\bB \otimes \cdots \otimes \cK^{-1}(s_{1})\bB \right),
	\end{align}
\end{subequations}
for $\xi \in \{2,\ldots,d\}$ and $\eta \in \{1,\ldots,d-1\}$, where $\cK^{-1}(s)$ is the Laplace transform of the fundamental solution $\mphi$. We have listed in \Cref{tab:LinearStructureExamples} some examples of the structure of $\cK(s)$ for certain types of structured systems. In this paper, we refer to the functions in \eqref{eq:general_TF}  as the \emph{multivariate transfer functions} associated with the polynomial structured system \eqref{eq:struct_nonlin_sys}.

\subsection{Interpolation-based MOR}
In this subsection, we present the construction of projection matrices $\bV$ and $\bW$, yielding ROMs \eqref{eq:ROMPolynomialOpt} such that the generalized transfer functions of the original model and ROM match at pre-defined interpolation points. The results presented here extend those presented in  \cite{morBenG21} from polynomial systems with $\cK(s) = s\bE-\bA$ and those derived in  \cite{Wer21} from quadratic-bilinear structured systems to the class of  structured polynomial systems. 


\begin{theorem}\label{thm:gen_interpolation}
	Let a SISO polynomial structured system be given as in \eqref{eq:struct_nonlin_sys}. Assume $\sigma_i$ and $\mu_i$, $i \in\{ 1,\ldots,\tilde r\}$, to be interpolation points such that $\cK(s)$ is invertible for all $s = \{\sigma_i,\mu_i\}$, $i \in \{1,\ldots,\tilde r\}$. Moreover, we define the projection matrices $\bV$ and $\bW$  as follows:
	\begin{subequations}\allowdisplaybreaks
	\begin{align*}
	\cV_{\mathrm L}	&=	\range{\cK^{-1}(\sigma_1)\bB,\ldots, \cK^{-1}(\sigma_{\tr})\bB},\\
	\cV_{\mathrm N} 	&=	\bigcup_{\eta= 1}^{d-1} \bigcup_{i= 1}^{\tr}  \range{ \cK^{-1}(\sigma_i) \bN_\eta\left(\cK^{-1}(\sigma_i)\bB \otimes \cdots \otimes \cK^{-1}(\sigma_i)\bB \right)},  \\
	\cV_{\mathrm H} 	&=	\bigcup_{\xi= 2}^d \bigcup_{i= 1}^{\tr}  \range{ \cK^{-1}(\sigma_i) \bH_\xi \left(\cK^{-1}(\sigma_i)\bB \otimes \cdots \otimes \cK^{-1}(\sigma_i)\bB \right) },\\ 
	\cW_{\mathrm L}	&=	\range{\cK^{-\top}(\mu_1)\bC^\top,\ldots, \cK^{-\top}(\mu_{\tr})\bC^\top},
	\\
	\cW_{\mathrm N} 	&=	\bigcup_{\eta= 1}^{d-1} \bigcup_{i= 1}^{\tr}  \textnormal{range}\big( \cK^{-\top}(\sigma_{i}) \left(\bN_\eta \right)_{(2)}\big(\cK^{-1}(\sigma_i)\bB  \otimes \cdots \otimes \cK^{-1}(\sigma_i)\bB \otimes\cK^{-\top}(\mu_i)\bC^\top\big)\big), 
	\\
	\cW_{\mathrm H} 	&=	\bigcup_{\xi= 2}^d  \bigcup_{i= 1}^{\tr}  \textnormal{range}\Big(\cK^{-\top}(\sigma_{i}) \left(\bH_\xi\right)_{(2)} \left(\cK^{-1}(\sigma_i)\bB \otimes \cdots \otimes \cK^{-\top}(\mu_i)\bC^\top \right)\Big) ,\\
	\range{\bV} 	& =	\cV_{\mathrm L}+ \cV_{\mathrm N}+ \cV_{\mathrm H},\\
	\range{\bW}	&= 	\cW_{\mathrm L}+ \cW_{\mathrm N}+\cW_{\mathrm H},
	\end{align*}
\end{subequations}
	where
	$\left(\bH_\xi\right)_{(2)} \in \R^{n\times n^\xi}$ and $\left(\bN_\eta\right)_{(2)} \in \R^{n\times m\cdot n^\xi}$ are, respectively, the mode-2 matricizations of the  $(\xi{+}1)$-way symmetric tensor $\bten{H}_\xi \in \R^{n\times \cdots \times n}$ and $(\eta{+}2)$-way symmetric  tensor $\bten{N}_\eta \in \R^{n\times \cdots \times n}$ whose mode-1 matricizations are $\bH_\xi$  and $\bN_\eta$, respectively.  Assume $\bV$ and $\bW$ are of full column rank and  $(\bF_{\mathrm L},\bF_{\mathrm N}, \bF_{\mathrm H})$ and $(\hat\bF_{\mathrm L},\hat\bF_{\mathrm N}, \hat\bF_{\mathrm H})$ are the multivariate transfer functions of the original model and ROM, respectively.  If the ROM is constructed by Petrov-Galerkin projections as in \eqref{eq:RedMatrices} using the matrices $\bV$ and $\bW$, then the ROM satisfies the following interpolation conditions:
	\begin{subequations}
		\begin{align}
			\bF_{\mathrm L}(\sigma_i)					&= 	\hat{\bF}_{\mathrm L}(\sigma_i), \label{eq:linear1} \\
			\bF_{\mathrm L}(\mu_i)						&= 	\hat{\bF}_{\mathrm L}(\mu_i), \label{eq:linear2}\\	
			\bF_{\mathrm N}^{(\eta)}(\sigma_i,\ldots,\sigma_i) 	&= 	\hat \bF_{\mathrm N}^{(\eta)}(\sigma_i,\ldots,\sigma_i),\label{eq:bilinear1} \\
			\bF_{\mathrm N}^{(\eta)}(\sigma_i,\ldots,\sigma_i,\mu_i) 	&=	\hat \bF_{\mathrm N}^{(\eta)}(\sigma_i,\ldots,\sigma_i,\mu_i), \label{eq:bilinear2}\\
			\bF_{\mathrm H}^{(\xi)}(\sigma_i,\ldots,\sigma_i) 	&= 	\hat \bF_H^{(\xi)}(\sigma_i,\ldots,\sigma_i), \label{eq:Qbilinear1}\\
			\bF_{\mathrm H}^{(\xi)}(\sigma_i,\ldots,\sigma_i,\mu_i) 	&= 	\hat \bF_{\mathrm H}^{(\xi)}(\sigma_i,\ldots,\sigma_i,\mu_i)\label{eq:Qbilinear2},
		\end{align}
	\end{subequations}
	provided the reduced matrix $ \chK(s) := \bW^\top\cK(s)\bV$ is non-singular for $s = \{\sigma_i,\mu_i\}$, $i \in \{1,\ldots,\tilde r\}$.
\end{theorem}
\begin{proof} 
	The theorem can be proven along the same lines as done in \cite{morBenG21}, where a special case with $\cK(s) = s\bE-\bA$ is considered.  However, we here provide the idea by providing the proof of the relation~\eqref{eq:bilinear1} only. 
	
	Firstly notice that, from \eqref{eq:RedMatrices}, we have $\chK(s) = \bW^\top\cK(s)\bV.$	The interpolation conditions in \eqref{eq:linear1} and \eqref{eq:linear2} follow directly from the structured linear case, see~\cite{morBeaG09}. Additionally,  one can easily prove the following relations:
	\begin{subequations}\allowdisplaybreaks\label{eq:linearRel}
		\begin{align}
			\bV\chK^{-1}(\sigma_i)\bhB 	&= \cK^{-1}(\sigma_i)\bB,  &\sigma_i &\in \{\sigma_1,\ldots,\sigma_{\tilde r}\},\label{eq:linearRel_V}\\
			\bhC\chK^{-1}(\mu_i)\bW^\top &= \bC\cK^{-1}(\mu_i),  &\mu_i &\in \{\mu_1,\ldots,\mu_{\tilde r}\}.\label{eq:linearRel_W}
		\end{align}
	\end{subequations}
	Next, let us prove the interpolation condition in \eqref{eq:bilinear1}. We begin with
	\begin{subequations}\allowdisplaybreaks
		\begin{align}
			&\bV\chK^{-1}(\sigma_i)\bhN_\eta\left(\chK^{-1}(\sigma_i)\bhB \otimes \cdots \otimes \chK^{-1}(\sigma_i)\bhB \right) \nonumber
			\\
			&\qquad\qquad = \bV\chK^{-1}(\sigma_i)\bW^\top\bN_\eta \bV^{\circled{\tiny{$\eta$}}}\left(\chK^{-1}(\sigma_i)\bhB \otimes \cdots \otimes \chK^{-1}(\sigma_i)\bhB \right)\nonumber
			\\
			&\qquad\qquad = \bV\chK^{-1}(\sigma_i)\bW^\top\bN_\eta \left( \bV \chK^{-1}(\sigma_i)\bhB \otimes \cdots \otimes \bV\chK^{-1}(\sigma_i)\bhB \right)\nonumber
			\\
			&\qquad\qquad = \bV\chK^{-1}(\sigma_i)\bW^\top\bN_\eta \left(  \cK^{-1}(\sigma_i)\bB \otimes \cdots \otimes \cK^{-1}(\sigma_i)\bB \right)\nonumber
			\\
			&\pushright{\left(\text{using}~ \ref{eq:linearRel_V}\right)}\nonumber
			\\
			&\qquad\qquad = \bV\chK^{-1}(\sigma_i)\bW^\top \cK(\sigma_i) \underbrace{\cK^{-1}(\sigma_i) \bN_\eta \left(  \cK^{-1}(\sigma_i)\bB \otimes \cdots \otimes \cK^{-1}(\sigma_i)\bB \right)}_{\in\range{\bV}\left(\therefore ~=: \bV\bz\right)}\nonumber
			\\
			&\qquad\qquad = \bV\chK^{-1}(\sigma_i)\bW^\top \cK(\sigma_i) \bV\bz, \nonumber
			\\
			& \qquad\qquad =\bV\chK^{-1}(\sigma_i) \chK(\sigma_i)\bz = \bV\bz, \label{eq:VPhi}
		\end{align}
	\end{subequations}
	where $\bz$ is a vector such that  $\bV\bz = \cK^{-1}(\sigma_i)\bB \otimes \cdots \otimes \cK^{-1}(\sigma_i)\bB$. Hence, pre-multiplying with  $\bC$ yields the interpolation condition \eqref{eq:bilinear1}. The proof of the interpolation condition \eqref{eq:Qbilinear1} follows in a similar way. Moreover, the other interpolation conditions can be proven similarly.  \qed
\end{proof}

\Cref{thm:gen_interpolation} allows to construct a ROM satisfying interpolation conditions of the multivariate system transfer functions \eqref{eq:general_TF}. 
In what follows, we show that if the right and left interpolation points are equal, i.e., $\sigma_i = \mu_i, i \in \{1,\ldots,\tilde{r}\}$, then Hermite interpolation conditions are also satisfied.

\begin{theorem}\label{theo:Hermite_intep} Under the hypothesis of \Cref{thm:gen_interpolation}, assume that $\sigma_i = \mu_i$, for $i \in\{ 1,\ldots,\tilde r\}$. Then,
	\begin{subequations}
		\begin{align}
			\dfrac{\partial}{\partial s_1} \bF_{\mathrm L}(\sigma_i)					&= 	\dfrac{\partial}{\partial s_1} \hat{\bF}_{\mathrm L}(\sigma_i), \label{eq:devlinear} 
			\\
			\dfrac{\partial}{\partial s_j} \bF_{\mathrm N}^{(\eta)}(\sigma_i,\ldots,\sigma_i) 	
			&= 	\dfrac{\partial}{\partial s_j} \hat \bF_{\mathrm  N}^{(\eta)}(\sigma_i,\ldots,\sigma_i),\quad \text{for}~\, j \in\{1, \dots, \eta+1 \}, \label{eq:devbilinear} 
			\\
			\dfrac{\partial}{\partial s_k} \bF_{\mathrm H}^{(\xi)}(\sigma_i,\ldots,\sigma_i) 	&= 	\dfrac{\partial}{\partial s_k} \hat \bF_{\mathrm H}^{(\xi)}(\sigma_i,\ldots,\sigma_i), \quad \text{for}~\, k \in\{1, \dots, \xi+1 \}. \label{eq:devQbilinear}
		\end{align}
	\end{subequations}
\end{theorem}

\begin{proof}
	First, note that the result \eqref{eq:devlinear} is well-known in the literature, see, e.g.,  \cite[Thm. 1]{morBeaG09}.  Thus, we focus here on  proving the interpolation condition \eqref{eq:devbilinear}.  Observe that 
	\[\dfrac{\partial}{\partial s} \cK^{-1}(s) = -\cK^{-1}(s)\left(\dfrac{\partial}{\partial s}\cK(s)\right) \cK^{-1}(s).\]
	Hence,
	\begin{subequations}\allowdisplaybreaks
		\begin{align*}
			\dfrac{\partial}{\partial s_1}& \hat \bF_{\mathrm N}^{(\eta)}(\sigma_i,\ldots,\sigma_i) 
			\\
			&=  \bhC \chK^{-1}(\sigma_i) \bhN_\eta\left(\chK^{-1}(\sigma_i)\bhB \otimes \cdots \otimes \dfrac{\partial}{\partial s_1}\chK^{-1}(\sigma_i)\bhB \right) 
			\\
			&= - \bhC \chK^{-1}(\sigma_i) \bhN_\eta\left(\chK^{-1}(\sigma_i)\bhB \otimes \cdots \otimes \chK^{-1}(\sigma_i)
			\left(\dfrac{\partial}{\partial s_1}\chK(\sigma_i)\right) \chK^{-1}(\sigma_i)\bhB \right) 
			\\
			&\pushright{\left(\text{using}~ \eqref{eq:matricizationRelation}\right)}
			\\
			&= - 
			\bhB^\top \chK^{-\top}(\sigma_i)\left(\dfrac{\partial}{\partial s_1}\chK(\sigma_i)\right)^\top \chK^{-\top}(\sigma_i)
			\left(\bhN_\eta\right)_{(2)}\left(\chK^{-1}(\sigma_i)\bhB \otimes \cdots \otimes  \chK^{-T}(\sigma_i)\bhC^\top \right) 
			\\ 
			&= -  \bhB^\top \chK^{-\top}(\sigma_i)\left(\dfrac{\partial}{\partial s_1}\chK(\sigma_i)\right)^\top \chK^{-\top}(\sigma_i) \bV^\top\left(\bN_\eta\right)_{(2)}\left(\bV\otimes \cdots\otimes \bV \otimes \bW\right)  \\
			& \pushright{\times \left(\chK^{-1}(\sigma_i)\bhB \otimes \cdots \otimes  \chK^{-\top}(\sigma_i)\bhC^\top \right) }
			\\ 
			&= -  \bhB^\top \chK^{-\top}(\sigma_i)\left(\dfrac{\partial}{\partial s_1}\chK(\sigma_i)\right)^\top \chK^{-\top}(\sigma_i) \bV^\top\left(\bN_\eta\right)_{(2)}\left(\cK^{-1}(\sigma_i)\bB \otimes \cdots \otimes  \cK^{-\top}(\sigma_i)\bC^\top \right) 
			\\
			&= -  \bhB^\top \chK^{-\top}(\sigma_i)\left(\dfrac{\partial}{\partial s_1}\chK(\sigma_i)\right)^\top \chK^{-\top}(\sigma_i) \bV^\top \cK(\sigma_i)^\top
			\\
			&\pushright{ \times \underbrace{\cK(\sigma_i)^{-\top}\left(\bN_\eta\right)_{(2)}\left(\cK^{-1}(\sigma_i)\bB \otimes \cdots \otimes  \cK^{-\top}(\sigma_i)\bC^\top \right)}_{\in\range{\bW}\left(\therefore ~=: \bW\bz\right)} }
			\\
			&= -  \bhB^\top \chK^{-\top}(\sigma_i)\left(\dfrac{\partial}{\partial s_1}\chK(\sigma_i)\right)^\top \chK^{-\top}(\sigma_i) \bV\cK(\sigma_i)^\top \bW z 
			\\ 
			&= -  \bhB^\top \chK^{-\top}(\sigma_i)\left(\dfrac{\partial}{\partial s_1}\chK(\sigma_i)\right)^\top  z 
			\\
			&= -  \bhB^\top \chK^{-\top}(\sigma_i) \bV^\top\left(\dfrac{\partial}{\partial s_1}\cK(\sigma_i)\right)^\top\bW  z 
			\\
			&= -  \bB^\top \cK^{-\top}(\sigma_i)\left(\dfrac{\partial}{\partial s_1}\cK(\sigma_i)\right)^\top\bW  z 
			\\
			&= -  \bB^\top \cK^{-\top}(\sigma_i)\left(\dfrac{\partial}{\partial s_1}\cK(\sigma_i)\right)^\top \cK(\sigma_i)^{-\top}\left(\bN_\eta\right)_{(2)}\left(\cK^{-1}(\sigma_i)\bB \otimes \cdots \otimes  \cK^{-\top}(\sigma_i)\bC^\top \right)
			\\
			& = \bC \cK^{-1}(\sigma_i) \bN_\eta\left(\cK^{-1}(\sigma_i)\bB \otimes \cdots \otimes \dfrac{\partial}{\partial s_1}\cK^{-1}(\sigma_i)\bB \right) 
			\\
			& = 	\dfrac{\partial}{\partial s_1}\bF_{\mathrm N}^{(\eta)}(\sigma_i,\ldots,\sigma_i).
		\end{align*}	
	\end{subequations}
	Furthermore, due to the symmetric tensors and \Cref{lemma:sym_tensor}, we have\[ \dfrac{\partial}{\partial s_1} \hat \bF_{\mathrm N}^{(\eta)}(\sigma_i,\ldots,\sigma_i) = \dfrac{\partial}{\partial s_j} \hat \bF_{\mathrm N}^{(\eta)}(\sigma_i,\ldots,\sigma_i),\quad j \in \{2,\ldots,\eta\}. \] Moreover, along the same lines as above, we can also show that 
	\begin{equation}
		\dfrac{d}{ds_{\eta+1}} \bF_{\mathrm N}^{(\eta)}(\sigma_i,\ldots,\sigma_i) = \dfrac{d}{ds_{\eta+1}} \hat\bF_{\mathrm N}^{(\eta)}(\sigma_i,\ldots,\sigma_i).
	\end{equation}
	Hence, the relations \eqref{eq:devbilinear} are proven, and similarly, the relations in \eqref{eq:devQbilinear} can also be proven.
\end{proof}

From  \Cref{thm:gen_interpolation,theo:Hermite_intep},  we are able to determine Petrov-Galerkin matrices $\bV$ and $\bW$, allowing to construct a desired interpolatory ROM. These results are derived up to now for the SISO case. The results generalize \cite[Thms. 5 and 6]{morBenGW20}.
Moreover, they can be easily generalized for the MIMO case using the idea of the so-called tangential interpolation. We discuss the MIMO result in \Cref{appen:TangInterp}.

%% file: paramMOR.tex
\section{Extension to Parametric  and Special Cases}\label{sec:par_interpolation}
In this section, we discuss extensions of the result presented in the previous section to parametric and some special cases. We begin with the parametric case.
\subsection{Parametric case}
Here, we consider parametric systems of the form:
\begin{subequations}\label{eq:struct_nonlin_sys_para}
	\begin{align}
		\left(\cL(\bp) \bx\right)(t,\bp)	&= \cP( \bx(t,\bp), \bu(t),\bp) +\bB(\bp) \bu(t), \label{eq:struct_nonlin_sys_dyn_para} 
		\\ 
		\by(t,\bp) 		  &= \bC(\bp) \bx(t,\bp), \label{eq:struct_nonlin_sys_obs_para} 
	\end{align}
\end{subequations}
where $\bp \in \mathrm\Omega \subset \R^q$ contains the system parameters; $\cL(\bp)$ is a parameterized linear operator; $\bB(\bp)\in \Rnm, \bC(\bp) \in \Rqn$ are parameter-dependent matrices, and the nonlinear term $\cP( \bx(t,\bp), \bu(t),\bp)$ takes the form:
\begin{equation}\label{eq:PolynomialOptPar}
	\cP(\bx(t,\bp), \bu(t),\bp) = \sum_{\xi = 2}^d \bH_{\xi}(\bp) \kronF{x}{$\xi$}(t,\bp) + \sum_{\eta = 1}^{d-1} \bN_{\eta}(\bp)\left(\bu(t)\otimes \bx^{\circled{\tiny {$\eta$}}}(t,\bp)\right).
\end{equation}
The parametric case results can be obtained by following the same lines as for the non-parametric case and discussion in \cite{morBenG19} for non-structured systems; thus, here we only briefly sketch the ideas.  Here again, for simplicity, we present the results for the SISO case. In the MIMO case, we can use the idea of tangential interpolation as discussed in \Cref{appen:TangInterp}.

Similar to the non-parametric case, we can derive generalized transfer functions for the parametric system, which are given as follows:
\begin{subequations}\label{eq:pargeneral_TF}
	\begin{align}
		\bF_{\mathrm L}(s_1,\bp) &:=  \bC(\bp)\cK^{-1}(s_1,\bp)\bB(\bp),
		\\
		\bF^{(\xi)}_{\mathrm H}(s_1,\ldots,s_{\xi+1}, \bp) & := \nonumber \\ \bC(\bp)\cK^{-1}(s_{\xi+1},&\bp) \bH_{\mathrm \xi}(\bp) \left(\cK^{-1}(s_\xi,\bp)\bB(\bp) \otimes \cdots \otimes \cK^{-1}(s_{1},\bp)\bB(\bp)\right),
		\\
		\bF_{\mathrm N}^{(\eta)}(s_1,\ldots, s_{\eta+1},\bp) &:=  \nonumber \\
		\bC(\bp) \cK^{-1}(s_{\eta+1},&\bp) \bN_\eta(\bp)\left(\cK^{-1}(s_{\eta},\bp)\bB(\bp) \otimes \cdots \otimes \cK^{-1}(s_{1},\bp)\bB(\bp) \right),
	\end{align}
\end{subequations}
where $\cK^{-1}(s,\bp)$ is the Laplace transform of the fundamental solution of $\cL(\bp)$.  If a ROM is computed using the projection matrices $\bV$ and $\bW$, assuming $\bV$ and $\bW$ are full rank matrices, the reduced operator and matrices are given as follows:
\begin{equation}\label{eq:Paracompute_rom}
	\begin{aligned}
		\hat\cK(s,\bp)& = \bW^\top\chK(s,\bp)\bV, ~~  \bhN_{\eta}(\bp) = \bW^\top\bN_{\eta}(\bp)\bV^{\circled{\tiny{$\eta$}}},~~ \eta\in \{1,\ldots,d-1\},
		\\
		\bhB(\bp) &= \bW^\top\bB(\bp), ~~   \bhC(\bp)= \bC(\bp)\bV, ~~ \bhH_{\xi}(\bp) = \bW^\top\bH_{\xi}(\bp)\bV^{\circled{\tiny{$\xi$}}},~~ \xi \in \{2,\ldots,d\}, 
	\end{aligned}
\end{equation}
and the inverse Laplace transform of $\hat\cK^{-1}(s,\bp)$ is the fundamental solution related to the reduced operator $\hat\cL(\bp)$.
In the following, an extension of \Cref{thm:gen_interpolation} to the parametric case is presented that allows to construct an interpolatory ROM. Interpolation-based MOR for structured parametric bilinear systems has been investigated in \cite{morBenGW2020par}, which we generalize to more general nonlinear systems. 

\begin{theorem}\label{thm:parametric_interpolation}
	Consider the original polynomial parametric system \eqref{eq:struct_nonlin_sys_para}, together with its multivariate transfer functions  given in \eqref{eq:pargeneral_TF}. Let $\sigma_i,\bp_i$ and $\mu_i$, $i \in\{ 1,\ldots,\tilde r\}$, be interpolation points such that $\cK(s,\bp)$ is invertible for all $s\in \{\sigma_i,\mu_i\}$, $i \in \{1,\ldots,\tilde r\}$, $\bp \in \{\bp_1,\ldots, \bp_{\tr}\}$. Moreover, let $\bV$ and $\bW$ be defined as follows: 
	\begin{subequations}\allowdisplaybreaks
		\begin{align*}
			\cV_{\mathrm L}&= \bigcup_{i= 1}^{\tr} \range{\cK^{-1}(\sigma_i,\bp_i)\bB(\bp_i)},
			\\
			\cV_{\mathrm N} &= \bigcup_{\eta= 1}^{d-1}\bigcup_{i= 1}^{\tr} \range{ \cK^{-1}(\sigma_i,\bp_i) \bN_\eta(\bp_i)\left(\cK^{-1}(\sigma_i,\bp_i)\bB(\bp_i) \otimes \cdots \otimes \cK^{-1}(\sigma_i,\bp_i)\bB(\bp_i)\right)},  
			\\
			\cV_{\mathrm H} &=  \bigcup_{\xi= 2}^d\bigcup_{i= 1}^{\tr}  \range{ \cK^{-1}(\sigma_i,\bp_i) \bH_\xi(\bp_i) \left(\cK^{-1}(\sigma_i,\bp_i)\bB(\bp_i)\otimes \cdots \otimes \cK^{-1}(\sigma_i,\bp_i)\bB(\bp_i)\right) },
			\\ 
			\cW_{\mathrm L}&= \bigcup_{i= 1}^{\tr}\range{\cK^{-\top}(\mu_i,\bp_i)\bC(\bp_i)^\top}
			\\
			\cW_{\mathrm N} &= \bigcup_{\eta= 1}^{d-1} \bigcup_{i= 1}^{\tr}\ensuremath{\mathop{\mathrm{range}}\left( \cK^{-1}(\sigma_{i},\bp_i) \left(\bN_\eta(\bp_i) \right)_{(2)}\left(\cK^{-1}(\sigma_i,\bp_i)\bB(\bp_i) \otimes \cdots \right.\right.}
			\\
			&\hspace{5cm}  \otimes \cK^{-1}(\sigma_i,\bp_i)\bB(\bp_i) \otimes\cK^{-\top}(\mu_i,\bp_i)\bC(\bp_i)^\top\Big), 
			\\
			\cW_{\mathrm H} &= \bigcup_{\xi= 2}^d \bigcup_{i= 1}^{\tr}\ensuremath{\mathop{\mathrm{range}}\left(\cK^{-1}(\sigma_{i},\bp_i) \left(\bH_\xi(\bp_i)\right)_{(2)} \left(\cK^{-1}(\sigma_{i},\bp_i)\bB(\bp_i) \otimes \cdots \right.\right.}\\
			&\hspace{5cm}  \otimes \cK^{-1}(\sigma_{i},\bp_i)\bB(\bp_i)\otimes \cK^{-\top}(\mu_i,\bp_i)C(\bp_i)^\top \Big), \\
			\range{\bV} &= \cV_{\mathrm L}+ \cV_{\mathrm N}+ \cV_{\mathrm H},\\
			\range{\bW}&= \cW_{\mathrm L}+ \cW_{\mathrm N}+\cW_{\mathrm H},
		\end{align*}
	\end{subequations}
	where $\left(\bH_\xi (\bp)\right)_{(2)} \in \R^{n\times n^\xi}$ and $\left(\bN_\eta (\bp)\right)_{(2)} \in \R^{n\times m\cdot n^\xi}$ are, respectively, the mode-2 matricizations of the $(\xi{+}1)$-way symmetric tensor $\bten{H}_\xi (\bp)\in \R^{n\times \cdots \times n}$ and $(\eta{+}2)$-way symmetric tensor $\bten{N}_\eta(\bp) \in \R^{n\times \cdots \times n}$ whose mode-1 matricizations are $\bH_\xi(\bp)$  and $\bN_\eta(\bp)$, respectively.
	Then, the following interpolation conditions are fulfilled:
	\begin{subequations}\allowdisplaybreaks
		\begin{align*}
			\bF_{\mathrm L}(\sigma_i,\bp_{i}) &= \hat{\bF}_{\mathrm L}(\sigma_i,\bp_{i}), 
			\\
			\bF_{\mathrm L}(\mu_i,\bp_{i})&= \hat{\bF}_{\mathrm L}(\mu_i,\bp_{i}), 
			\\
			\bF_{\mathrm N}^{(\eta)}(\sigma_i,\ldots,\sigma_i,\bp_{i})&= 	\hat \bF_{\mathrm N}^{(\eta)}(\sigma_i,\ldots,\sigma_i, \bp_{i}),  
			\\
			\bF_{\mathrm N}^{(\eta)}(\sigma_i,\ldots,\sigma_i, \mu_i,\bp_{i})&= 	\hat{\bF}_{\mathrm N}^{(\eta)}(\sigma_i,\ldots,\sigma_i,\mu_i,\bp_{i}) 
			\\
			\bF_{\mathrm H}^{(\xi)}(\sigma_i,\ldots,\sigma_i,\bp_{i}) &= 	\hat{\bF}_{\mathrm H}^{(\xi)}(\sigma_i,\ldots,\sigma_i,\bp_{i}), 
			\\
			\bF_{\mathrm H}^{(\xi)}(\sigma_i,\ldots,\sigma_i, \mu_i,\bp_{i})&= 	\hat{\bF}_{\mathrm H}^{(\xi)}(\sigma_i,\ldots,\sigma_i, \mu_i,\bp_{i}),
		\end{align*}
	\end{subequations}
	provided the reduced matrix $\bW^\top\cK(s, \bp)\bV$ is non-singular for $s \in \{\sigma_i,\mu_i\}$ and $\bp \in\{\bp_i\}$, $i \in \{1,\ldots,\tilde r\}$.
\end{theorem}
\begin{proof}
	The proof is analogous to the proof of \Cref{thm:gen_interpolation} and extends the proof of \cite[Thm. 2]{morBenGW2020par}. Therefore, for the brevity of the paper, we skip the proof.  
\end{proof}

Note that we have assumed a general parametric structure for the system matrices, e.g., $\cL(\bp)$ (or $\cK(s,\bp)), \bH_{\xi}(\bp)$. Then, the corresponding ROM  can be computed as shown in \eqref{eq:Paracompute_rom}, but it may be required to compute a ROM for each parameter setting. However, if we assume an affine parametric structure of the system matrices as follows:
\begin{equation}\label{eq:affineSys}
	\begin{aligned}
		\cK(s,\bp) &= \sum_{i=1}^{t_A}\kappa_i(s,\bp) \bA^{(i)},& \quad  \bB(\bp) &= \sum_{i=1}^{t_b} \alpha^{(i)}_b(\bp) \bB^{(i)},  \quad  \bC(\bp) = \sum_{i=1}^{t_c} \alpha^{(i)}_c(\bp) \bC^{(i)},  \\
		\bN_\eta(\bp) &= \sum_{i=1}^{t_{n_\eta}} \alpha^{(i)}_{n_\eta}(\bp) \bN_\eta^{(i)}, &\quad  \bH_\xi(\bp) &= \sum_{i=1}^{t_{h_\xi}} \alpha^{(i)}_{h_\xi}(\bp) \bH_\xi^{(i)},
	\end{aligned}
\end{equation}
then the resulting ROM with the same structure can be determined using
\begin{equation}\label{eq:parameter_reduced}
	\begin{aligned}
		\cK(s,\bp) &= \sum_{i=1}^{t_A}\kappa_i(s,\bp) \bhA^{(i)}, &\quad  \bhB(\bp) &= \sum_{i=1}^{t_b} \alpha^{(i)}_b(\bp) \bhB^{(i)},  \quad  \bhC(\bp) = \sum_{i=1}^{t_c} \alpha^{(i)}_c(\bp) \bhC^{(i)}, \\
		\bhN_\eta(\bp) &= \sum_{i=1}^{t_{n_\eta}} \alpha^{(i)}_{n_\eta}(\bp) \bhN_\eta^{(i)}, &\quad  \bhH_\xi(\bp) &= \sum_{i=1}^{t_{h_\xi}} \alpha^{(i)}_{h_\xi}(\bp) \bhH_\xi^{(i)},
	\end{aligned}
\end{equation}
where the original matrices are reduced by the standard projections
\begin{equation}\label{eq:parPetrovGalerk}
	\begin{aligned}
		\bhA^{(i)} &= \bW^\top\bA^{(i)}\bV, \quad\quad  \bhB^{(i)} = \bW^\top\bB^{(i)}, \quad\quad   \bhC^{(1)} = \bC^{(i)}\bV, \\
		\bhN_{\eta}^{(i)} &= \bW^\top\bN_{\eta}^{(i)}\bV^{\circled{\tiny{$\eta$}}}, \quad   \bhH_{\xi}^{(i)} = \bW^\top\bH_{\xi}^{(i)}\bV^{\circled{\tiny{$\xi$}}},
	\end{aligned}
\end{equation}
for $\eta\in \{1,\ldots,d-1\}$ and $\xi \in \{2,\ldots,d\}$. 
Consequently, we can pre-compute  these reduced matrices. Hence, the computation of a ROM for each parameter becomes numerically cheaper. 

\begin{remark} As in \Cref{theo:Hermite_intep}, if $\sigma_i = \mu_i$, Hermite interpolation conditions with respect to the parameter are satisfied, i.e., 
\begin{align*}
\grad_{\bp}\,\bF_{\mathrm L}(\sigma_i,\bp_{i}) &= \grad_{\bp}\,\hat{\bF}_{\mathrm L}(\sigma_i,\bp_{i}),
\\
\grad_{\bp}\, \bF_{\mathrm N}^{(\eta)}(\sigma_i,\ldots,\sigma_i,\bp_{i})&= 	\grad_{\bp}\,\hat \bF_{\mathrm N}^{(\eta)}(\sigma_i,\ldots,\sigma_i, \bp_{i}),  
\\
\grad_{\bp}\,\bF_{\mathrm H}^{(\xi)}(\sigma_i,\ldots,\sigma_i,\bp_{i}) &= 	\grad_{\bp}\,\hat{\bF}_{\mathrm H}^{(\xi)}(\sigma_i,\ldots,\sigma_i,\bp_{i}).
\end{align*}
Thus, the sensitivity of the multivariate transfer functions with respect to the parameter is preserved, which may be useful in, e.g., parameter optimization. Since this result can be proven in a similar way as the one in \Cref{theo:Hermite_intep} and \cite[Thm. 3]{morBenGW2020par}, we decide to omit it for the brevity of the paper.
\end{remark}

\subsection{Structured input and output matrices}
One may also consider the case where the input and output matrices $\bB$ and $\bC$ depend on the frequency $s$. This happens to be the case if the system is subject to input or/and output delays, e.g, $\cB(s) = \bB_1 +\bB_2e^{-\tau s}$.  In such scenarios, we can employ the results of \Cref{thm:gen_interpolation} by replacing $\bB$ by $\cB(\sigma_i)$ and $\bC$ by $\cC(\mu_i)$ everywhere, e.g., the matrices $\bV_{\mathrm L}$ and $\bW_{\mathrm L}$ in \Cref{thm:gen_interpolation} are given by
\begin{subequations}\allowdisplaybreaks
	\begin{align*}
		\bV_{\mathrm L}	&= 	\range{\cK^{-1}(\sigma_1)\cB(\sigma_1),\ldots, \cK^{-1}(\sigma_{\tr})\cB(\sigma_{\tr})},\\
		\bW_{\mathrm L}	&= 	\range{\cK^{-\top}(\mu_1)\cC^\top(\mu_1),\ldots, \cK^{-\top}(\mu_{\tr})\cC^\top(\mu_{\tr})}.
	\end{align*}
\end{subequations}
The interpolation results presented in this paper also hold in this context. Since their proof is a straightforward extension of the results presented up-to-now, we refrain from providing details. 

%% file: dropalgo.tex
	\section{Determining lower-order approximate interpolatory models}\label{sec:DROP}
	In the previous sections, we have stated  results for constructing interpolatory ROMs by a Petrov-Galerkin projection. In the proposed methodology, the quality of the ROMs highly depends on the choice of the interpolation points.  It is an open question of how to select these interpolation points optimally, and this remains an important problem to be investigated in the future. 

	However, inspired by the discussions in \cite{morIonA14,morAntGI16,morBenG21,morGosA18}, we propose a scheme based on a oversampling as follows. 
In this scheme, we first compute the projection matrices $\bV$ and $\bW$ by considering several interpolation points in a given domain. After that, we aim to determine the dominant subspaces that not only allow us to determine lower-order models but also approximately interpolate all the considered interpolation points. 
	To that end, let us assume to have a structured polynomial system as in \eqref{eq:struct_nonlin_sys} and the projection matrices $\bV$ and $\bW$ as defined in \Cref{thm:gen_interpolation}. For this section, we assume  that $\bV$ and $\bW$ contain the columns of the Krylov basis exactly for the given interpolation points. Furthermore, let us construct the matrices 
	\begin{equation}\label{eq:loewner_str}
		\begin{aligned}
			\tilde{\cK}(s) &= \bW^{\top}\cK(s)\bV, \quad  \tilde\bB = \bW^{\top}\bB, \quad   \tilde\bC= \bC\bV,\\
			\tilde\bH_{\xi} &= \bW^{\top}\bH_{\xi}\bV^{\circled{\tiny{$\xi$}}}, \quad  \tilde\bN_{\eta} = \bW^{\top}\bN_{\eta}\bV^{\circled{\tiny{$\eta$}}}.
		\end{aligned}
	\end{equation}
Next, we can extend the observation from \cite{morAntGI16,morBenG21,morGosA18}. This means that if the pencil $\tilde\cK(s)$ is regular, then the realization \eqref{eq:loewner_str} is a realization interpolating the original model. 

However, when we consider many interpolation points, then often the pencil  $\tilde\cK(s)$ becomes singular. In this case, there exists a lower-order realization that can interpolate at all the given interpolation points. To obtain such a realization, we follow the idea proposed in \cite{morAntGI16,morBenG21,morGosA18,morBenGP19}. For this, we first consider the form of  $\cK(s)$ to be
\begin{equation}
	\cK(s) = \alpha_1(s)\bA^{(1)} + \cdots + \alpha_l(s)\bA^{(l)},
\end{equation} 
and assume that 
\begin{equation}\label{eq:RankCondition}
	\rank{\begin{bmatrix}
			\bW^\top\bA^{(1)}\bV , \dots , \bW^\top\bA^{(l)}\bV
	\end{bmatrix}} = \rank{\begin{bmatrix}
			\bW^\top\bA^{(1)}\bV \\ \vdots \\ \bW^\top\bA^{(l)}\bV
	\end{bmatrix}} = \hr.
\end{equation}
Then there exists a structured system of order $\hr\leq r$, realizing the model whose generalized transfer functions also interpolate at the pre-defined interpolation points. Consequently,  using~\eqref{eq:RankCondition}, we can estimate the complexity of the underlying dynamical system.  Additionally, an SVD procedure based on the matrices in \eqref{eq:RankCondition} allows us to construct a ROM using appropriate subspaces. 

To obtain the corresponding subspaces and the ROM, we propose  \Cref{algo:Inter_structured_para}, enabling construction of ROMs for polynomial structured systems \eqref{eq:struct_nonlin_sys}. This can be seen as an extension of \cite[Algo. 3]{morBenG21} to structured polynomial systems.  The procedure consists in selecting interpolation points $\sigma_i$, $ i \in \{1,\ldots,\tilde r\}$, and constructing the matrices $\bV$ and $\bW$ as in \Cref{thm:gen_interpolation} (steps 2 and 3) containing the original Krylov basis.  Then, in step 4, we compute the SVDs of the matrices in \eqref{eq:SVDalgo}. As  discussed earlier, the numerical rank of these matrices  indicates the order of a good ROM, and the left and right singular vectors allow us to determine dominant subspaces. Hence, in step 5, the projection matrices $\bV_{\mathrm e}$ and $\bW_{\mathrm e}$ are constructed. Finally, in step 6, the ROM  is computed in the framework of Petrov-Galerkin projection. 

\Cref{algo:Inter_structured_para} can easily be adapted to parametric and MIMO cases by computing the projection matrices $\bV$ and $\bW$ appropriately in step 3 of the algorithm.  In this case, we need to specify parameters within the range of interest and tangential directions along with the interpolation points for the frequency variable.

\begin{algorithm}[tb]
	\caption{Construction of ROMs for Structured Polynomial Systems.}\label{algo:Inter_structured_para}
	\begin{algorithmic}[1]
		\State \textbf{Input:} The polynomial structured system matrices as in \eqref{eq:affineSys}, and order of a ROM~$r$.
		\State Choose left and right interpolation points. 
		\State Compute $\bV$ and $\bW$ using the interpolation points as in \Cref{thm:gen_interpolation}.
		\State Determine SVDs  \begin{equation}\label{eq:SVDalgo}
		\begin{bmatrix}
		\bW^\top\bA^{(1)}\bV, \dots, \bW^\top\bA^{(l)}\bV
		\end{bmatrix} = \bW_1 \mathrm\Sigma_l \tilde{\bV}^\top \quad \text{and} \quad  \begin{bmatrix}
		\bW^\top\bA^{(1)}\bV \\ \vdots \\ \bW^\top\bA^{(l)}\bV
		\end{bmatrix} = \tilde{\bW}\mathrm\Sigma_r \bV_1^\top 
		\end{equation} \label{step:SVD}
		\State Compute projection matrices: 
		$\bV_{\mathrm e} = \bV\bV_{1}(:,1{:}r)$ and $\bW_{\mathrm e} = \bW\bW_{1}(:,1{:}r) $.
		\State Compute reduced matrices:
		\[
		\begin{aligned}
		 \chK(s) &= \bW_{\mathrm e}^{\top}\cK(s)\bV_{\mathrm e}, \quad  \bhB = \bW_{\mathrm e}^{\top}\bB, \quad   \bhC= \bC\bV_{\mathrm e},\\
		\bhH_{\xi} &= \bW_{\mathrm e}^{\top}\bH_{\xi}\bV_{\mathrm e}^{\circled{\tiny{$\xi$}}}, \quad  \bhN_{\eta} = \bW_{\mathrm e}^{\top}\bN_{\eta}\bV_{\mathrm e}^{\circled{\tiny{$\eta$}}},
				\end{aligned}
		\]
		\State \textbf{Output:} The reduced-order matrices: $\hat{\cK}(s),\hat{\bB}$, $\hat{\bC}$, $\bhH_{\xi}$  and $\bhN_{\eta}$.
	\end{algorithmic}
\end{algorithm}

\begin{remark}\label{rem:Galerkin} 
	In many applications, such as  dynamical systems with symmetric matrices, or systems with a dissipative realization, it is desirable to apply Galerkin projections, i.e., to enforce $\bW = \bV$ (see, e.g., \cite{morCasLE12}). As a result, the reduced systems would potentially preserve system properties, e.g., dissipativity, symmetry. Although \Cref{algo:Inter_structured_para} is designed for Petrov-Galerkin projection, one can still adapt it to allow only Galerkin  projection. Indeed, in step 3 one just needs to compute $\bV$ and set $\bW = \bV$. Hence, in step 5, $\bW_{\mathrm e} = \bV_{\mathrm e}$. 
\end{remark}

\begin{remark} 
	It is worth mentioning that one does not need to use any hyper reduction method for the fast computation of the nonlinear terms. Indeed, as shown in~\eqref{eq:RedMatrices}, the nonlinear terms are already directly computed using the Petrov-Galerkin projection. Additionally, at a first glance, the computation of the reduced matrix $\bhH_{\xi} = \bW^{\top}\bH_{\xi}\bV^{\circled{\tiny{$\xi$}}}$ 
	seems to be numerically expensive as one needs to evaluate $\bV^{\circled{\tiny{$\xi$}}}$. However, the authors in \cite{morBenG19} proposed a procedure based on Hadamard product form of the term $\bH_{\xi}x^{\circled{\tiny{$\xi$}}} = \cA_1 x \circ \dots \circ\cA_{\eta}x $. As a consequence, the computation of $\bhH_{\xi}$ can be performed without the explicit computation of $\bV^{\circled{\tiny{$\xi$}}}$. 
\end{remark}

%% file: numerics_new.tex
\section{Numerical examples}\label{sec:Numerics}
In this section, we illustrate the efficiency of the proposed method by means of three examples. We also compare with the recent interpolation method proposed in \cite{morBenGW2020par,morBenGW20} for two of the examples. To integrate nonlinear structured systems, we use an explicit Euler scheme, and for non-structured systems, we use the function \texttt{ode15s} in \matlab. 
For the random number generator, we have used the seed `$0$'.  All experiments were performed using \matlab~(2020b) running on a Macbook Pro with 2,3 GHz 8-Core \intel~\coreinine~CPU, 16GB of RAM, and Mac OS X v10.15.6.

\subsection{Parametric Chafee-Infante equation}
In our first example,  we consider the one-dimensional parametric Chafee-Infante system governed by the following partial differential equation
\begin{equation*}
	\begin{aligned} 
		\dot{v}(t) &=v_{xx} +v(\bp-v^2),\quad x\in (0,1)\times (0,T), \quad &v(0,t) = u(t), \quad t\in(0,T),\\
		v_x(L,t) &= 0,\quad  t\in(0,T), &v(x,0) =0,\quad x\in(0,1), 
	\end{aligned}
\end{equation*}
where the parameter is assumed to lie in the interval $\bp \in \left[0.25,2\right]$. After a spatial discretization using a finite-difference method for a uniform grid with $k = 500$ points, one obtains a high-fidelity cubic parametric model of  the form:
\begin{equation}\label{eq:CI_Example}
	\begin{aligned}
		\dot{\bv}(t) &= \bA_1 \bv(t) + \bp\bA_p + \bH_3 \left(\bx(t)\otimes \bx(t) \otimes \bx(t)\right) 	+\bB \bu(t), \\
		\by(t) &= \bC \bx(t).
	\end{aligned}
\end{equation}  
The MOR problem for its non-parametric variant (for $\bp = 1$) has been considered in \cite{morBenG21}, where the polynomial non-linearity in the ROMs is preserved. Also, authors in, e.g., \cite{morBenB15,morBenG17,morBenGG18} have constructed ROMs of the model \eqref{eq:CI_Example}, but they do not preserve the cubic structure. These methods rewrite the model into quadratic-bilinear form, followed by employing MOR techniques for quadratic-bilinear systems. 

Here, we aim at constructing a reduced cubic parametric system using \Cref{algo:Inter_structured_para}. For this, we take $200$ points in the frequency range $\left[10^{-3},10^3\right]$ and the equal number of points for the parameter in the considered interval. 

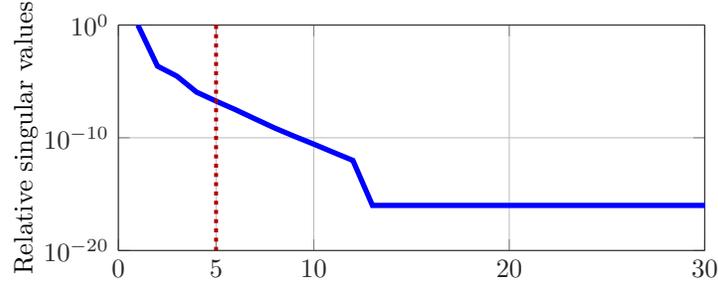
\begin{figure}[!tb]
	\centering
	\setlength\fheight{3.0cm}
	\setlength\fwidth{.5\textwidth}
	\input{Figures/ChafeeDecaSingularValuesPara.tikz}%

	\caption{Parametric Chafee-Infante equation: relative decay of singular values based on the Loewner pencil.}
	\label{figure:chafee_SVD_Para}
\end{figure}

First, in \Cref{figure:chafee_SVD_Para}, we plot the decay of the singular values based on the matrices in \eqref{eq:RankCondition} and we observe a very fast decay.  Subsequently, we determine a reduced parametric system of order $r = 5$. To compare the quality of the ROM, we simulate for the same inputs $u^{(1)}(t) = 10(\sin(\pi t)+1)$ and $u^{(2)}(t) = 5(te^{-t}) $ and  parameters $\bp = \{0.25,1,2\}$. We plot the transient response and relative errors in \Cref{fig:chafee_input1_para,fig:chafee_input2_Para}, illustrating that the reduced parametric system can capture the dynamics of the high-fidelity model very well for different inputs and different parameters. 

\begin{figure}[!tb]
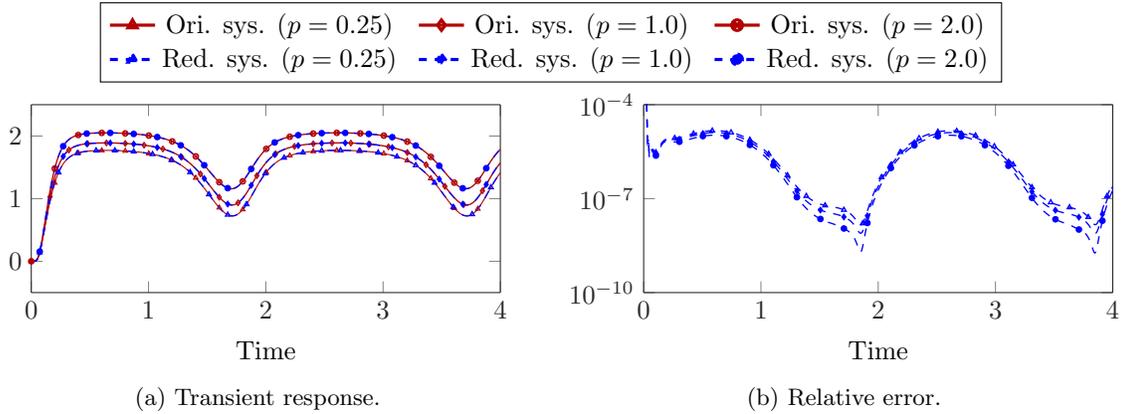

	\centering
	\begin{tikzpicture}
		\begin{customlegend}[legend columns=3, legend style={/tikz/every even column/.append style={column sep=0.2cm}} , legend entries={Ori. sys. $(p = 0.25)$,Ori. sys. $(p = 1.0)$,Ori. sys. $(p = 2.0)$,Red. sys. $(p = 0.25)$, Red. sys. $(p = 1.0)$,  Red. sys. $(p = 2.0)$}, ]
			\addlegendimage{color=black!30!red,solid,line width=1.2pt,mark = triangle}
			\addlegendimage{color=black!30!red,solid,line width=1.2pt,mark = diamond}
			\addlegendimage{color=black!30!red,solid,line width=1.2pt,mark = otimes}
			\addlegendimage{color=blue,solid,line width=1.2pt,mark = triangle*,dashed}
			\addlegendimage{color=blue,solid,line width=1.2pt,mark = diamond*,dashed}
			\addlegendimage{color=blue,solid,line width=1.2pt,mark = otimes*,dashed}    \end{customlegend}
	\end{tikzpicture}    
	\begin{subfigure}[t]{.5\textwidth}
		\centering
		\setlength\fheight{2.5cm}
		\setlength\fwidth{.8\textwidth}
	\input{Figures/ChafeeInput1_Response_Para.tikz}%
        
		\caption{Transient response.}
		\label{fig:chafee_input1_res_Para}
	\end{subfigure}%
	\begin{subfigure}[t]{.5\textwidth}
		\centering
		\setlength\fheight{2.5cm}
		\setlength\fwidth{.8\textwidth}
	\input{Figures/ChafeeInput1_RelErr_Para.tikz}%

		\caption{Relative error.}
		\label{fig:chafee_input1_err_Para}
	\end{subfigure}
	\caption{Parametric Chafee-Infante equation: a comparison of the original and ROM for the input $\bu^{(1)} = 10\left(\sin(\pi t)+1\right)$ and for different parameter values.}
	\label{fig:chafee_input1_para}
\end{figure}

\begin{figure}[!tb]
	\centering
	\begin{tikzpicture}
		\begin{customlegend}[legend columns=3, legend style={/tikz/every even column/.append style={column sep=0.2cm}} , legend entries={Ori. sys. $(p = 0.25)$,Ori. sys. $(p = 1.0)$,Ori. sys. $(p = 2.0)$,Red. sys. $(p = 0.25)$, Red. sys. $(p = 1.0)$,  Red. sys. $(p = 2.0)$}, ]
			\addlegendimage{color=black!30!red,solid,line width=1.2pt,mark = triangle,every mark/.append style={solid},}
			\addlegendimage{color=black!30!red,solid,line width=1.2pt,mark = diamond,every mark/.append style={solid},}
			\addlegendimage{color=black!30!red,solid,line width=1.2pt,mark = otimes,every mark/.append style={solid},}
			\addlegendimage{color=blue,solid,line width=1.2pt,mark = triangle*,dashed,every mark/.append style={solid},}
			\addlegendimage{color=blue,solid,line width=1.2pt,mark = diamond*,dashed,every mark/.append style={solid},}
			\addlegendimage{color=blue,solid,line width=1.2pt,mark = otimes*,dashed,every mark/.append style={solid},}    \end{customlegend}
	\end{tikzpicture}  
	\begin{subfigure}[t]{.5\textwidth}
		\centering
		\setlength\fheight{2.5cm}
		\setlength\fwidth{.8\textwidth}
	\input{Figures/ChafeeInput2_Response_Para.tikz}%
        
		\caption{Transient response.}
		\label{fig:chafee_input2_res_Para}
	\end{subfigure}%
	\begin{subfigure}[t]{.5\textwidth}
		\centering
		\setlength\fheight{2.5cm}
		\setlength\fwidth{.8\textwidth}
	\input{Figures/ChafeeInput2_RelErr_Para.tikz}%

		\caption{Relative error.}
		\label{fig:chafee_input2_err_Para}
	\end{subfigure}
	\caption{Parametric Chafee-Infante equation: a comparison of the original and ROM for the input $\bu^{(2)} = 5\left(e^{-t}t\right)$ and for different parameter values.}
	\label{fig:chafee_input2_Para}
\end{figure}

\subsection{Mechanical System}
In our second example, we consider a damped mass-spring system from \cite{morBenGW20}. The dynamics of the system is governed by a second-order bilinear system of the form:
\begin{equation*} 
	\begin{aligned}
		\bM \ddot{\bx}(t) + \bD \dot{\bx}(t) + \bK \bx(t) &=
		\bN_{1} \bx(t) \bu_{1}(t) + \bN_{2} \bx(t) \bu_{2}(t) + 	\bB\bu(t),\\
		\by(t) &= \bC \bx(t) \bx(t).
	\end{aligned}
\end{equation*}
We consider the same setting as provided in \cite[Sec. 4.2]{morBenGW20}. The order of the full model is $n = 1\,000$, and the model has two inputs and two outputs.  Mechanical systems can  have many structural properties such as passivity, and positive definiteness of mass, damping, and stiffness matrices, see, e.g., the survey paper~\cite{morBedBDetal20}. Therefore, it is desired to preserve these properties in the ROM. This can be achieved if the ROM is determined via Galerkin  (one-sided) projection instead of Petrov-Galerkin (two-sided) one. Hence, in this example, we construct a ROM using Galerkin projection via \Cref{algo:Inter_structured_para} as discussed in \Cref{rem:Galerkin}. 

\begin{figure}[!tb]
	\centering
	\setlength\fheight{3.0cm}
	\setlength\fwidth{.5\textwidth}
	\input{Figures/SO_SV.tikz}%

	\caption{Mechanical example: The decay of singular values obtained from \Cref{algo:Inter_structured_para}.}
	\label{fig:SO_SVD}
\end{figure}
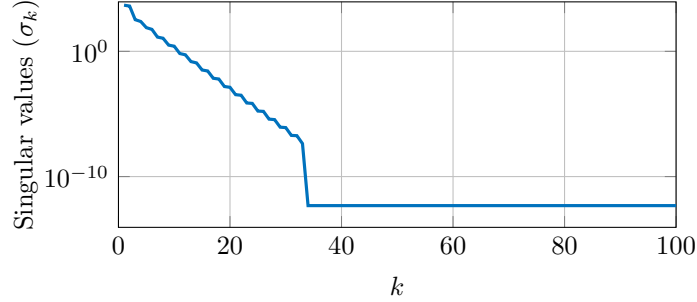

Next, to employ \Cref{algo:Inter_structured_para}, we consider $1\,000$ logarithmically distributed  frequency points on the imaginary axis in the range $\left[10^{-3},10^{3}\right]$. Since the system is MIMO, we also choose tangential directions which are taken randomly. We then observe the singular values obtained from \Cref{algo:Inter_structured_para} in \Cref{fig:SO_SVD}, indicating a sharp decay of the singular values.  From the figure, we note that the singular values after $33$ are at the level of machine precision. 

We determine ROMs of order $r = \{10,20,30\}$ (denoted by \texttt{StrDsp\_SO}). We compare the quality of the ROMs \texttt{StrDsp\_SO} with the method proposed in \cite{morBenGW20}. The ROM in \cite{morBenGW20} (denoted by \texttt{StrInt\_SO}) is computed based on interpolation.  
We choose the same interpolation points as in \cite{morBenGW20}, which yields $36$ basis vectors. To determine ROMs of order $r = \{10,20,30\}$, we take the same number of dominant basis vectors as the order out of $36$ basis vectors, and these dominant basis vectors are determined based on the QR decomposition of the basis vectors.

Next, to assess the quality of both ROMs, we compare transient responses of them with the full model using a control input $\bu(t) = 50\cdot \begin{bmatrix} \sin(20t) +1 \\ \sin(t)e^{-0.1t} \end{bmatrix}$. This is shown in~\Cref{fig:SO_timedomain}. Moreover, \Cref{tab:SO_outputerrors} shows a comparison of the $L_2$ and $L_\infty$-errors these three different ROMs. We observe that our method constructs ROMs which are consistently better both in $L_2$ and $L_\infty$-norm for all orders.

Finally, we would like to remark on computational aspects. Although we could obtain a better reduced-order model using \texttt{StrDsp\_SO} (approximately two orders of magnitude better for smaller orders), it comes with computational expenses. The philosophy of \texttt{StrDsp\_SO} relies on considering many interpolation points and, after that, on compression to determine global dominant subspaces. On the other hand, \texttt{StrInt\_SO} considers carefully choosing a few interpolation points. Consequently, \texttt{StrDsp\_SO} becomes much more computationally expensive as compared to \texttt{StrDsp\_SO}; for this example, $1~000$ interpolation points are considered for \texttt{StrDsp\_SO}, where only $3$ interpolation points are considered for \texttt{StrInt\_SO}. In the future, we will investigate active-learning-based approaches to sample a few points for \texttt{StrDsp\_SO} so that only relevant interpolation points (e.g., from $1~000$) are considered to determine the global dominant subspaces.

\definecolor{mycolor1}{rgb}{0.00000,0.44700,0.74100}%
\definecolor{mycolor2}{rgb}{0.00000,0.44700,0.74100}%
\definecolor{mycolor3}{rgb}{0.92900,0.69400,0.12500}%
\definecolor{mycolor4}{rgb}{0.92900,0.69400,0.12500}%
\definecolor{mycolor5}{rgb}{0.46600,0.67400,0.18800}%
\definecolor{mycolor6}{rgb}{0.46600,0.67400,0.18800}%

\begin{figure}[!tb]
	\centering
	\begin{tikzpicture}
		\begin{customlegend}[legend columns=3, legend style={/tikz/every even column/.append style={column sep=0.2cm}} , legend entries={Ori. sys. $(n = 1\,000)$, \texttt{StrInt\_SO} $(r=10)$ \cite{morBenGW20}, \texttt{StrDsp\_SO} $(r=10)$}, ]
			\addlegendimage{color=mycolor1,solid,line width = 1.2pt}
			\addlegendimage{color=mycolor5,solid,line width = 1.2pt,dotted}
			\addlegendimage{color=mycolor3,solid,line width = 1.2pt,dashed}
		\end{customlegend}
	\end{tikzpicture}  
	\begin{subfigure}[t]{0.48\textwidth}
		\centering
		\setlength\fheight{3.0cm}
		\setlength\fwidth{0.9\textwidth}
		\input{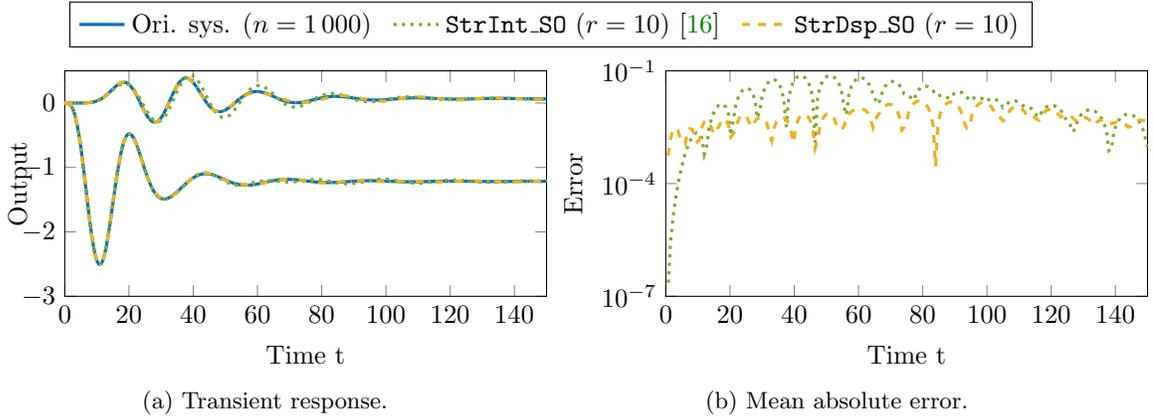}
		\caption{Transient response.}
	\end{subfigure}
	\begin{subfigure}[t]{0.48\textwidth}
		\centering
		\setlength\fheight{3.0cm}
		\setlength\fwidth{0.9\textwidth}
		\input{Figures/SO_timedomain_err10.tikz}
		\caption{Mean absolute error.}
	\end{subfigure}
	\caption{Mechanical example: A comparison of the transient responses the original and ROMs of order $r=10$ for the control input $\bu(t) = 50\cdot \protect\begin{bmatrix} \sin(20t) +1 \\ \sin(t)e^{-0.1t}  \protect\end{bmatrix}$.}
	
	\label{fig:SO_timedomain}
\end{figure}

	\renewcommand{\arraystretch}{1.3}
\begin{table}[]
	\centering
	\begin{tabular}{|c|c|c||c|c|} \hline
		&  \multicolumn{2}{c}{$L_2$-error} &  \multicolumn{2}{c}{$L_\infty$-error} \\ \hline
	 & \texttt{StrInt\_SO} & \texttt{StrDsp\_SO}    & \texttt{StrInt\_SO} & \texttt{StrDsp\_SO}      \\ \hline
	  $r  =10$&      $4.68 \cdot 10^{-4} $      &    $1.01\cdot 10^{-4}$             &  $ 7.48\cdot 10^{-2}$   & $ 1.59\cdot 10^{-2}$    \\ \hline
$r  =20$&   $1.06\cdot 10^{-5}$         &           $6.31\cdot 10^{-6}$ & $1.63\cdot 10^{-3}$ &  $7.65\cdot 10^{-4}$   \\ \hline            
$r  =30$&   $8.59\cdot 10^{-8}$         &  $3.83\cdot 10^{-8}$    &  $2.16\cdot 10^{-5}$&  $3.98\cdot 10^{-6}$     \\ \hline      
	\end{tabular}
	\caption{Mechanical example: The $L_2$  and $L_\infty$-errors of the outputs between the original model and ROMs.}
\label{tab:SO_outputerrors}
\end{table}

\subsection{Parametric bilinear time-delay system}%
\label{subsec:delayexample}
In our last experiment, we consider an example from \cite{morBenGW2020par,morGosPBetal19} that models a time-delayed heated rod using a one-dimensional parametric heat equation:
\begin{equation}\label{eq:pde_delay}
	\partial_{t} v(x, t)  = \partial_{x}^2 v(x, t)
	-\bp\sin(x) v(x, t)
	+ \bp \sin(x) v(x, t - 1) + u(t),
\end{equation}
with homogeneous Dirichlet boundary conditions and $\bp \in [1,10]$.  Spatial discretization yields a parametric bilinear systems of the form:
\begin{equation}\label{eq:delay_model}
	\begin{aligned}
		\dot{\bx}(t) & = \bA(\bp) \bx(t) +		\bp \bA_{\mathrm{d}} \bx(t - 1) + \bN \bx(t) \bu(t) + \bB \bu(t), \\
		\by(t) & =  \bC \bx(t),
	\end{aligned}
\end{equation}
where $\bA(\bp) = \bA_{0} - \bp \bA_{\mathrm{d}}$.  We have employed $5\,000$ grid points to discretize the PDEs \eqref{eq:pde_delay}, thus leading to the state-space model \eqref{eq:delay_model} of order $n = 5\,000$. Additionally, this model is single-input single-output. The example  also completely fits in our set-up, discussed in \Cref{sec:setup}.  In this case,  
\begin{equation*}
	\begin{aligned}
		\cK(s, \bp) & = s\bI_{n} - (\bA_{0} - \bp \bA_{\mathrm{d}})
		- \bp e^{-s} \bA_{\mathrm{d}},\quad
		\bB(\bp)  = \bB,\\
		\bN(\bp) &= \bN,\quad \bH(\bp) = 0,\quad \text{and}~\bC(\bp) = \bC.
	\end{aligned}
\end{equation*}
Next, we aim at constructing a ROM using \Cref{algo:Inter_structured_para}. In order to employ the algorithm, we consider $1\,000$ logarithmically distributed frequency points in the range $\left[10^{-2},10^2\right]$, and for the parameter $\bp$, we randomly take the parameter in the considered parameter range. Consequently, we have tuples $\{\sigma_i,\bp_i\}$, $i \in\{1,\ldots,1\,000\}$. 

Next, we plot the decay of the singular values obtained from \Cref{algo:Inter_structured_para} in \Cref{figure:delay_SVD}, indicating a rapid decay. Hence, we can expect a good quality ROM of small order. Then, we  determine ROMs of order $r = \{5,10,20\}$ using \Cref{algo:Inter_structured_para}  (denoted by (\texttt{StrDsp\_delay})). We compare the quality of the (\texttt{StrDsp\_delay}) with a ROM obtained using the interpolation-based methodology proposed in \cite{morBenGW2020par} (denoted by \texttt{StrInt\_delay}). We use the interpolation points as in \cite{morBenGW2020par}, which yield $24$ basis vectors.  We construct \texttt{StrInt\_delay} ROMs of order $r = \{5,10,20\}$ using the dominant basis vectors of the $24$ basis vectors, which is done by taking QR decomposition of the basis vectors.


\begin{figure}[!tb]
	\centering
	\setlength\fheight{3.0cm}
	\setlength\fwidth{.5\textwidth}
	\input{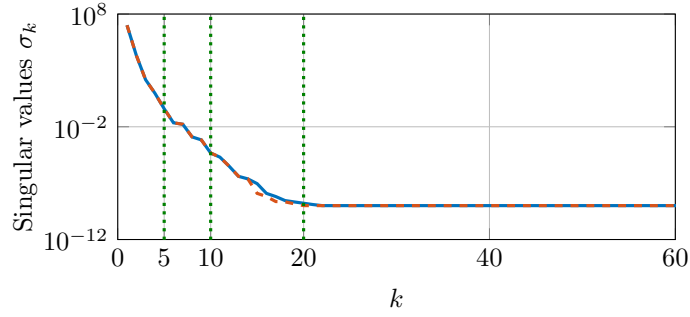}%

	\caption{Delay parametric example: The decay of singular values obtained from \Cref{algo:Inter_structured_para}.}
	\label{figure:delay_SVD}
\end{figure}

Next, we compare the time-domain simulation of both ROMs with the high-fidelity model for a wide range of parameters. We use  the same control input as  in~\cite{morBenGW2020par,morGosPBetal19}; that is $u(t)  = 0.05 \left( \cos(10t) + \cos(5t) \right)$, and vary the parameter in the range of interest. To perform time-domain simulation, we employ the explicit Euler method with time stepping $dt = 2\cdot 10^{-2}$. To measure the quality of the ROMs, we consider the following error function:
\begin{equation}\label{eq:delay_err}
	\cE(t, \bp) := \frac{\left\| \by(t; \bp) - \hat\by(t; \bp) \right\|}
	{ \max\limits_t\max\limits_{\bp \in [1,10]}\left\| \by(t; \bp) \right\| }.
\end{equation}
Then, we plot the transient responses and errors (as defined in \eqref{eq:delay_err}) for different parameters in \Cref{fig:delay_timedomain}. Furthermore, we compute  the maximum error in the time and parameter domain as follows:
\begin{equation*}
	\cE_{\max}:=	\max\limits_{\bp \in [1, 10]} \left( \max\limits_{t \in [0,10]}
	\cE(t, \bp) \right),
\end{equation*} 
which is reported in \Cref{tab:delay_error} (see the first columns of the table). 
%

We notice that  both ROMs perform equally good, but \texttt{StrDsp\_delay} is consistently better than \texttt{StrInt\_delay}. In particular, the proposed methodology can produce high-quality parametric ROMs of small orders as it determines the dominant subspace jointly for parameter and frequency in the given domains.  

\begin{figure}[!tb]
	\centering
	\setlength\fheight{3.0cm}
	\setlength\fwidth{.5\textwidth}
	\begin{subfigure}[t]{\textwidth}
		\centering
		\includegraphics[width = \textwidth]{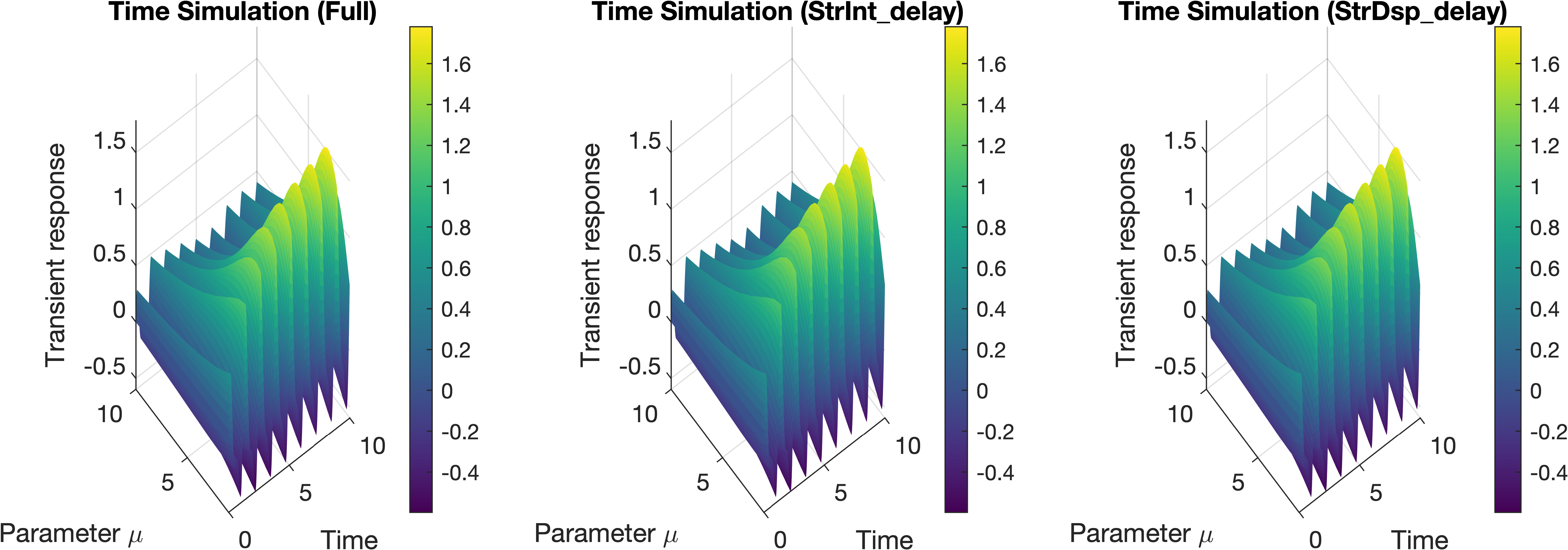}
		\caption{Transient response.}
	\end{subfigure}
	\begin{subfigure}[t]{\textwidth}
		\centering
		\includegraphics[width = 0.8\textwidth]{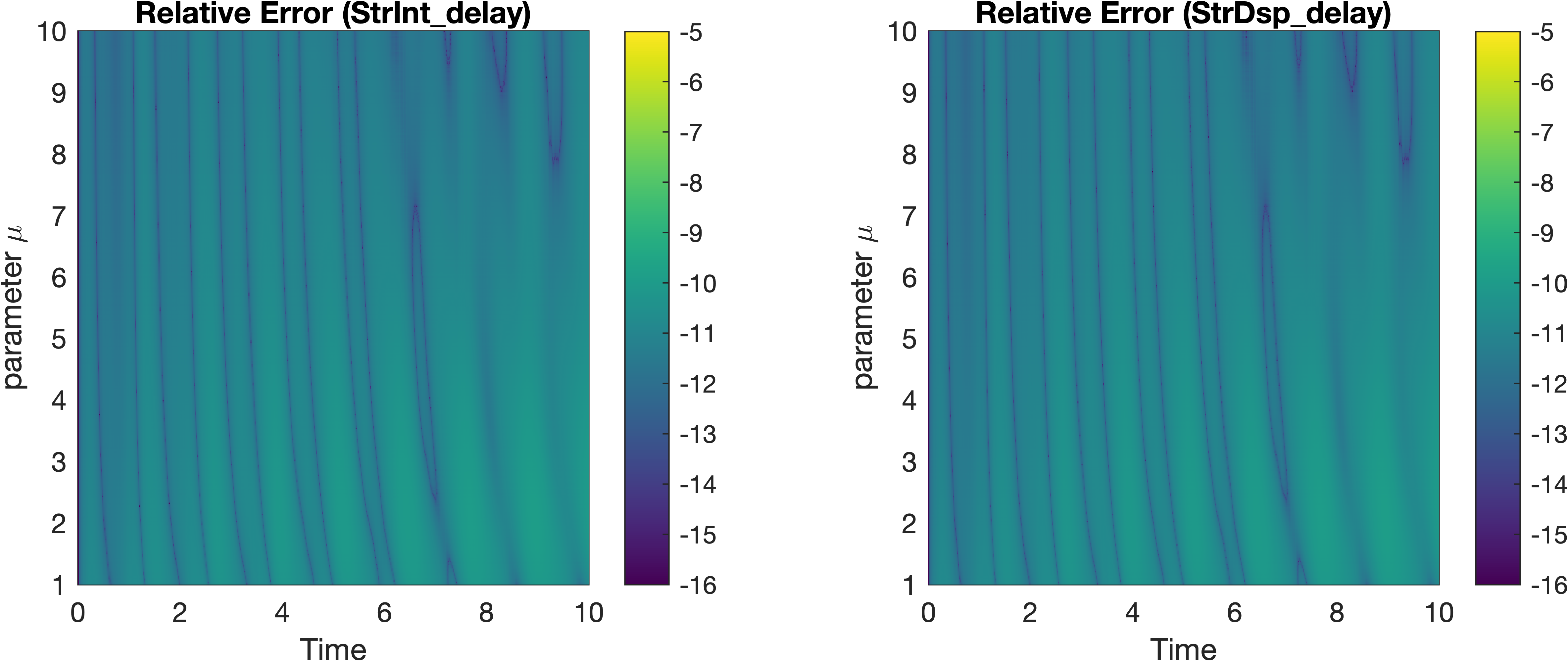}
		\caption{Relative error.}
	\end{subfigure}
	\caption{Delay parametric example: A comparison of transient response for an input $u(t) = 0.05 \left( \cos(10t) + \cos(5t) \right) $.}
	
	\label{fig:delay_timedomain}
\end{figure}


\begin{table}[tb]
	\centering
	\begin{tabular}{|c|c|c|} \hline
		& \texttt{StrInt\_SO} & \texttt{StrDsp\_SO}      \\ \hline
		$r  =5$&      $8.45\cdot 10^{-6} $      &    $1.00\cdot 10^{-6}$         \\ \hline
		$r  =10$&   $1.87\cdot 10^{-8}$         &           $3.25 \cdot 10^{-9}$  \\ \hline            
		$r  =20$&   $1.38\cdot 10^{-10}$         &  $1.39\cdot 10^{-10}$   \\ \hline      
	\end{tabular}
	\caption{Delay parametric example: Error between the full model and ROMs, namely \texttt{StrInt\_delay} and \texttt{StrDsp\_delay} of different orders.}
	\label{tab:delay_error}
\end{table}


%% file: Figures/ChafeeDecaSingularValuesPara.tikz
%
%
%
\definecolor{mycolor1}{rgb}{0.00000,0.44700,0.74100}%
\begin{tikzpicture}

\begin{axis}[%
width=\fwidth,
height=\fheight,
scale only axis,
separate axis lines,
every outer x axis line/.append style={white!15!black},
every x tick label/.append style={font=\color{white!15!black}},
xmin=0,
xmax=30,
every outer y axis line/.append style={white!15!black},
every y tick label/.append style={font=\color{white!15!black}},
ymode=log,
ymin=1e-20,
ymax=1,
xtick = {0,5,10,20,30},
grid = major, 
yminorticks=true,
ylabel = {Relative singular values},
]
\addplot [color=blue,line width = 2pt]
  table[row sep=crcr]{1	1\\
2	0.000227405444096018\\
3	2.88110124409778e-05\\
4	1.12050883443126e-06\\
5	1.77639451667461e-07\\
6	3.05513124815707e-08\\
7	4.66532125082632e-09\\
8	7.31170141745193e-10\\
9	1.3598454824944e-10\\
10	2.72984750996104e-11\\
11	5.10059718656763e-12\\
12	9.83143145454434e-13\\
13	9.99200722162641e-17\\
14	9.99200722162641e-17\\
15	9.99200722162641e-17\\
16	9.99200722162641e-17\\
17	9.99200722162641e-17\\
18	9.99200722162641e-17\\
19	9.99200722162641e-17\\
20	9.99200722162641e-17\\
21	9.99200722162641e-17\\
22	9.99200722162641e-17\\
23	9.99200722162641e-17\\
24	9.99200722162641e-17\\
25	9.99200722162641e-17\\
26	9.99200722162641e-17\\
27	9.99200722162641e-17\\
28	9.99200722162641e-17\\
29	9.99200722162641e-17\\
30	9.99200722162641e-17\\
31	9.99200722162641e-17\\
32	9.99200722162641e-17\\
33	9.99200722162641e-17\\
34	9.99200722162641e-17\\
35	9.99200722162641e-17\\
36	9.99200722162641e-17\\
37	9.99200722162641e-17\\
38	9.99200722162641e-17\\
39	9.99200722162641e-17\\
40	9.99200722162641e-17\\
41	9.99200722162641e-17\\
42	9.99200722162641e-17\\
43	9.99200722162641e-17\\
44	9.99200722162641e-17\\
45	9.99200722162641e-17\\
46	9.99200722162641e-17\\
47	9.99200722162641e-17\\
48	9.99200722162641e-17\\
49	9.99200722162641e-17\\
50	9.99200722162641e-17\\
};
\addplot [color=black!30!red,dotted, mark options={solid},mark size = 1.5,line width = 1.5pt]
  table[row sep=crcr]{5	1\\
  5 1e-20\\
  };
\end{axis}
\end{tikzpicture}%

%% file: Figures/SO_SV.tikz
%
\definecolor{mycolor1}{rgb}{0.00000,0.44700,0.74100}%
\begin{tikzpicture}

\begin{axis}[%
width=0.951\fwidth,
height=\fheight,
at={(0\fwidth,0\fheight)},
scale only axis,
xmin=0,
xmax=100,
ymode=log,
ymin=8.86023529873289e-15,
ymax=8860.2352987329,
yminorticks=true,
axis background/.style={fill=white},
xmajorgrids,
ymajorgrids,
yminorgrids,
xlabel = {$k$},
ylabel = {Singular values ($\sigma_k$)},
legend style={legend cell align=left, align=left, draw=white!15!black}
]
\addplot [color=mycolor1, line width = 1.2pt]
  table[row sep=crcr]{%
1	4813.9376637967\\
2	4091.10110906624\\
3	336.270003259226\\
4	239.216310342766\\
5	77.4828942614593\\
6	56.4257731867225\\
7	13.6245840011276\\
8	11.0229063336737\\
9	3.1212238524053\\
10	2.55159596688714\\
11	0.663697835565845\\
12	0.527870112125615\\
13	0.151814620738479\\
14	0.118606189041056\\
15	0.032194723742485\\
16	0.0269700785367365\\
17	0.0070526892196888\\
18	0.00613869802694754\\
19	0.00154605769752436\\
20	0.00134158014303899\\
21	0.000346890991858145\\
22	0.000310769113173767\\
23	7.52917570243549e-05\\
24	6.93162046605202e-05\\
25	1.70315119393293e-05\\
26	1.60040839144292e-05\\
27	3.87107039098806e-06\\
28	3.52652266539145e-06\\
29	8.79581691916005e-07\\
30	8.13023862405971e-07\\
31	1.9393289000195e-07\\
32	1.8399334167382e-07\\
33	4.31688149283455e-08\\
34	4.8100899901116e-13\\
35	4.8100899901116e-13\\
36	4.8100899901116e-13\\
37	4.8100899901116e-13\\
38	4.8100899901116e-13\\
39	4.8100899901116e-13\\
40	4.8100899901116e-13\\
41	4.8100899901116e-13\\
42	4.8100899901116e-13\\
43	4.8100899901116e-13\\
44	4.8100899901116e-13\\
45	4.8100899901116e-13\\
46	4.8100899901116e-13\\
47	4.8100899901116e-13\\
48	4.8100899901116e-13\\
49	4.8100899901116e-13\\
50	4.8100899901116e-13\\
51	4.8100899901116e-13\\
52	4.8100899901116e-13\\
53	4.8100899901116e-13\\
54	4.8100899901116e-13\\
55	4.8100899901116e-13\\
56	4.8100899901116e-13\\
57	4.8100899901116e-13\\
58	4.8100899901116e-13\\
59	4.8100899901116e-13\\
60	4.8100899901116e-13\\
61	4.8100899901116e-13\\
62	4.8100899901116e-13\\
63	4.8100899901116e-13\\
64	4.8100899901116e-13\\
65	4.8100899901116e-13\\
66	4.8100899901116e-13\\
67	4.8100899901116e-13\\
68	4.8100899901116e-13\\
69	4.8100899901116e-13\\
70	4.8100899901116e-13\\
71	4.8100899901116e-13\\
72	4.8100899901116e-13\\
73	4.8100899901116e-13\\
74	4.8100899901116e-13\\
75	4.8100899901116e-13\\
76	4.8100899901116e-13\\
77	4.8100899901116e-13\\
78	4.8100899901116e-13\\
79	4.8100899901116e-13\\
80	4.8100899901116e-13\\
81	4.8100899901116e-13\\
82	4.8100899901116e-13\\
83	4.8100899901116e-13\\
84	4.8100899901116e-13\\
85	4.8100899901116e-13\\
86	4.8100899901116e-13\\
87	4.8100899901116e-13\\
88	4.8100899901116e-13\\
89	4.8100899901116e-13\\
90	4.8100899901116e-13\\
91	4.8100899901116e-13\\
92	4.8100899901116e-13\\
93	4.8100899901116e-13\\
94	4.8100899901116e-13\\
95	4.8100899901116e-13\\
96	4.8100899901116e-13\\
97	4.8100899901116e-13\\
98	4.8100899901116e-13\\
99	4.8100899901116e-13\\
100	4.8100899901116e-13\\
};

\end{axis}

\begin{axis}[%
width=1.227\fwidth,
height=1.227\fheight,
at={(-0.16\fwidth,-0.135\fheight)},
scale only axis,
xmin=0,
xmax=1,
ymin=0,
ymax=1,
axis line style={draw=none},
ticks=none,
axis x line*=bottom,
axis y line*=left,
legend style={legend cell align=left, align=left, draw=white!15!black}
]
\end{axis}
\end{tikzpicture}%

%% file: Figures/SO_timedomain_err10.tikz
%
\definecolor{mycolor1}{rgb}{0.00000,0.44700,0.74100}%
\definecolor{mycolor1}{rgb}{0.46600,0.67400,0.18800}%

\definecolor{mycolor2}{rgb}{0.85000,0.32500,0.09800}%
\definecolor{mycolor2}{rgb}{0.92900,0.69400,0.12500}%
\begin{tikzpicture}

\begin{axis}[%
width=0.951\fwidth,
height=\fheight,
at={(0\fwidth,0\fheight)},
scale only axis,
xmin=0,
xmax=150,
ymode=log,
ymin=1e-07,
ymax=0.1,
xlabel = {Time t},
ylabel = {Error},
yminorticks=true,
ylabel style={rotate=0, at={(0.0,0.5)}},
axis background/.style={fill=white},
legend style={legend cell align=left, align=left, draw=white!15!black}
]
\addplot [color=mycolor1, line width = 1.2pt, dotted]
  table[row sep=crcr]{%
0	0\\
0.75	2.30177104170843e-07\\
1.5	2.817137266756e-06\\
2.25	1.22338390115066e-05\\
3	3.67572855028392e-05\\
3.75	8.93442430805874e-05\\
4.5	0.000187840067362697\\
5.25	0.000352369788410577\\
6	0.000597443257451885\\
6.75	0.00092734492970387\\
7.5	0.00132263515005165\\
8.25	0.00173565625813013\\
9	0.00208693141597652\\
9.75	0.00226958672534235\\
10.5	0.00215874946652566\\
11.25	0.00163425236307152\\
12	0.000608916046521982\\
12.75	0.000999426393834456\\
13.5	0.00304397583217028\\
14.25	0.00545103021078247\\
15	0.00800505898528693\\
15.75	0.0103686245905571\\
16.5	0.0121694561688789\\
17.25	0.0130280259552177\\
18	0.0126047091859635\\
18.75	0.010640342269667\\
19.5	0.00706243589173089\\
20.25	0.00193143006454011\\
21	0.00477734193191588\\
21.75	0.011692255193127\\
22.5	0.0188860603908918\\
23.25	0.0255415108263048\\
24	0.0307873354043995\\
24.75	0.0338884127146392\\
25.5	0.0342835738763492\\
26.25	0.0315607360260455\\
27	0.0256755894780425\\
27.75	0.0168341818875959\\
28.5	0.00555002064921647\\
29.25	0.00842145092162555\\
30	0.020792065775444\\
30.75	0.0336901119910921\\
31.5	0.0448525164087007\\
32.25	0.0531767316243804\\
33	0.057820401245769\\
33.75	0.0581419176303703\\
34.5	0.0539045228810019\\
35.25	0.0452831902474759\\
36	0.0327146443297772\\
36.75	0.0171293939673737\\
37.5	0.0031781831989629\\
38.25	0.0185104646031409\\
39	0.0358622041365045\\
39.75	0.0511000130290077\\
40.5	0.0630502221116369\\
41.25	0.0707789440094504\\
42	0.0736265414639029\\
42.75	0.0714265123218307\\
43.5	0.0642298383147883\\
44.25	0.0525778105783476\\
45	0.037313523578114\\
45.75	0.0194540997368708\\
46.5	0.00124973387912666\\
47.25	0.0188453728400279\\
48	0.0366929006238534\\
48.75	0.0520720603681331\\
49.5	0.0639653014208337\\
50.25	0.0716909097016706\\
51	0.074805534963491\\
51.75	0.0731702022402005\\
52.5	0.067068217239074\\
53.25	0.0569080881970928\\
54	0.0434792699759694\\
54.75	0.0277014181684885\\
55.5	0.0105812864756763\\
56.25	0.00676095479159035\\
57	0.0232848426994951\\
57.75	0.0380146973784406\\
58.5	0.050145916262854\\
59.25	0.0590088506947353\\
60	0.0642428707078725\\
60.75	0.0656402679471016\\
61.5	0.0632471796296852\\
62.25	0.0573929879253539\\
63	0.0484760550845325\\
63.75	0.0371684270452075\\
64.5	0.0241884884269138\\
65.25	0.0174153568321753\\
66	0.0129459385032945\\
66.75	0.01685414539104\\
67.5	0.0286830921752003\\
68.25	0.038526199923099\\
69	0.0459033865545433\\
69.75	0.0505582332211861\\
70.5	0.052325214014829\\
71.25	0.051237185214198\\
72	0.0475080329875952\\
72.75	0.0414041976469505\\
73.5	0.0333973248282235\\
74.25	0.0239781874864969\\
75	0.0225909144037473\\
75.75	0.0211629247941305\\
76.5	0.0189180462304003\\
77.25	0.0165344213347\\
78	0.0247121179179176\\
78.75	0.0312478214071339\\
79.5	0.0359064887633677\\
80.25	0.0385112066627129\\
81	0.0390348408124933\\
81.75	0.0375656442973772\\
82.5	0.0342417118005189\\
83.25	0.0293575822474038\\
84	0.023219440094503\\
84.75	0.0214416951020595\\
85.5	0.0223248340090166\\
86.25	0.0222438803856192\\
87	0.0212187246692198\\
87.75	0.0192946313947715\\
88.5	0.018264628476831\\
89.25	0.0226796205631747\\
90	0.0257267461918912\\
90.75	0.0273346121938331\\
91.5	0.0275043758950385\\
92.25	0.0262842262143791\\
93	0.0238367167375625\\
93.75	0.0203338818661134\\
94.5	0.0160139921217046\\
95.25	0.0174062350443805\\
96	0.0190410153506552\\
96.75	0.0198250630791417\\
97.5	0.0197378310760221\\
98.25	0.0188002804721362\\
99	0.0170889698069349\\
99.75	0.015097636785944\\
100.5	0.0170864063452341\\
101.25	0.0181145594776167\\
102	0.0181786428535078\\
102.75	0.0173522367199126\\
103.5	0.015717575468507\\
104.25	0.0134132241018119\\
105	0.0106010568861782\\
105.75	0.0117307184690054\\
106.5	0.0137370916031432\\
107.25	0.0151049013074467\\
108	0.0157997120393169\\
108.75	0.0158268087162191\\
109.5	0.0151985810047887\\
110.25	0.0139767669746647\\
111	0.0122317395989438\\
111.75	0.0111008174471331\\
112.5	0.0111409935232228\\
113.25	0.0106226267574753\\
114	0.00961750501591342\\
114.75	0.00821236661636212\\
115.5	0.0065014239253335\\
116.25	0.00583437806086282\\
117	0.00793598719655423\\
117.75	0.00964798564792712\\
118.5	0.010922504304097\\
119.25	0.0117196251627526\\
120	0.0120345423324108\\
120.75	0.0118696708687221\\
121.5	0.0112475963097417\\
122.25	0.0102164676597692\\
123	0.00882305499945322\\
123.75	0.00713963358029347\\
124.5	0.00529224894253891\\
125.25	0.00446471678676605\\
126	0.0034807260676143\\
126.75	0.00240183845175738\\
127.5	0.0029015815887762\\
128.25	0.00467990514781207\\
129	0.00621749587703361\\
129.75	0.00746737402831267\\
130.5	0.0083886643714848\\
131.25	0.00895669899061943\\
132	0.00916572546606728\\
132.75	0.00901297131014225\\
133.5	0.0085217957195694\\
134.25	0.00772001487030052\\
135	0.00664623632735883\\
135.75	0.00535529481574799\\
136.5	0.00389920942682062\\
137.25	0.00234176738111136\\
138	0.000745941202566998\\
138.75	0.000826300634161586\\
139.5	0.0023149512073115\\
140.25	0.00366499869806332\\
141	0.00482895377348015\\
141.75	0.00576984111643153\\
142.5	0.00645520889494081\\
143.25	0.00687123224336765\\
144	0.00700869756676698\\
144.75	0.00687075067760596\\
145.5	0.00647615060064955\\
146.25	0.00584399979992106\\
147	0.00501048048764617\\
147.75	0.00401405243811601\\
148.5	0.002897206891442\\
149.25	0.00170978107419955\\
150	0.000946556862021393\\
};

\addplot [color=mycolor2, dashed, line width = 1.2pt]
  table[row sep=crcr]{%
0	0\\
0.75	0.000555389252535343\\
1.5	0.00162263491837832\\
2.25	0.00256884731607719\\
3	0.00296040623933486\\
3.75	0.00261757722064944\\
4.5	0.00164170815509367\\
5.25	0.0012580951552539\\
6	0.00112759937865785\\
6.75	0.00228929037181715\\
7.5	0.00290169845030285\\
8.25	0.0027715119192079\\
9	0.00185649204705223\\
9.75	0.00128561747627682\\
10.5	0.00145697324552116\\
11.25	0.00316334035287994\\
12	0.0043784000581295\\
12.75	0.00481208125499863\\
13.5	0.00433092290273286\\
14.25	0.0029926235591444\\
15	0.0020821384860666\\
15.75	0.00264626882465221\\
16.5	0.0032482818710606\\
17.25	0.00471161939813261\\
18	0.00530574259722988\\
18.75	0.00492132136655346\\
19.5	0.00366922434972758\\
20.25	0.00174461660272382\\
21	0.00212795703806959\\
21.75	0.00281759693167434\\
22.5	0.00479466024992162\\
23.25	0.00619023601861941\\
24	0.00681218906861342\\
24.75	0.0065435495888129\\
25.5	0.00539351462903655\\
26.25	0.0035160970344941\\
27	0.00324917827806051\\
27.75	0.00453808989296339\\
28.5	0.00537591814146146\\
29.25	0.00559713053287003\\
30	0.00526392129064029\\
30.75	0.00497368107676059\\
31.5	0.00390378805996503\\
32.25	0.00224258795935704\\
33	0.00139142184634054\\
33.75	0.00312320527198354\\
34.5	0.00449563637149057\\
35.25	0.00537747678540626\\
36	0.00567095387658026\\
36.75	0.00562836852449161\\
37.5	0.00507890705237904\\
38.25	0.00397140270395249\\
39	0.00246962171132445\\
39.75	0.00123386761136809\\
40.5	0.00332560287235711\\
41.25	0.00511249919493344\\
42	0.00637340473333631\\
42.75	0.00697105295114129\\
43.5	0.00685076133782611\\
44.25	0.00605885460299186\\
45	0.00470309886370431\\
45.75	0.00291743395305806\\
46.5	0.000873611241805107\\
47.25	0.00144710963425698\\
48	0.00331025804546138\\
48.75	0.00509732676976254\\
49.5	0.00643960289273925\\
50.25	0.00718601056817666\\
51	0.00723045355990794\\
51.75	0.0065457976863501\\
52.5	0.00520959180287541\\
53.25	0.00334935327108986\\
54	0.00313510791820823\\
54.75	0.00272253711771831\\
55.5	0.00336928093328932\\
56.25	0.00537509633173653\\
57	0.0070350307593528\\
57.75	0.00824878561641856\\
58.5	0.0089355036262142\\
59.25	0.00902449735279777\\
60	0.00866127165379327\\
60.75	0.00939487169508624\\
61.5	0.009461976710925\\
62.25	0.00879436573525499\\
63	0.00736927884876347\\
63.75	0.00523982986461663\\
64.5	0.00310480783067989\\
65.25	0.00501804336407006\\
66	0.00658506247515382\\
66.75	0.00773905755800279\\
67.5	0.00935283589241386\\
68.25	0.0113332596620774\\
69	0.0125314570544934\\
69.75	0.0128940267294509\\
70.5	0.0123760800431793\\
71.25	0.0109993266885848\\
72	0.00884594946646832\\
72.75	0.00600087585709587\\
73.5	0.00264653670573891\\
74.25	0.00315674145585761\\
75	0.00480315326146743\\
75.75	0.00823724309551196\\
76.5	0.0112757660380338\\
77.25	0.0136493735648937\\
78	0.0152210505085558\\
78.75	0.0159142896392404\\
79.5	0.0157182636272132\\
80.25	0.0146595361129205\\
81	0.0128005544659879\\
81.75	0.0102333200395434\\
82.5	0.00708523245158311\\
83.25	0.00353245952880082\\
84	0.000284408184111143\\
84.75	0.00391353451481091\\
85.5	0.00734443905855631\\
86.25	0.0102792644471227\\
87	0.0125687670355003\\
87.75	0.0141065559827577\\
88.5	0.0148300920959163\\
89.25	0.0147515767661117\\
90	0.0138642609531772\\
90.75	0.0122309241006963\\
91.5	0.00993530209171423\\
92.25	0.00707246887339607\\
93	0.00395765255685985\\
93.75	0.0029681824182576\\
94.5	0.0031027513556527\\
95.25	0.00634233261758892\\
96	0.0091935500875582\\
96.75	0.0115109597619353\\
97.5	0.0132020947228883\\
98.25	0.0142116938629524\\
99	0.0145177375367329\\
99.75	0.0141163592118919\\
100.5	0.0130285045291379\\
101.25	0.0113059642577563\\
102	0.0090415602423009\\
102.75	0.00662756084852655\\
103.5	0.00622057019473456\\
104.25	0.00555919817872468\\
105	0.00465239660636221\\
105.75	0.00501874592100725\\
106.5	0.00726013077194104\\
107.25	0.00904108965912431\\
108	0.0103002030222752\\
108.75	0.0110094023090749\\
109.5	0.0111220584529472\\
110.25	0.0106543678676991\\
111	0.00962608193463902\\
111.75	0.00809088814869098\\
112.5	0.00617016899497054\\
113.25	0.00635200157437522\\
114	0.00625674297402925\\
114.75	0.00585599712944049\\
115.5	0.0051695247391421\\
116.25	0.00481258025983632\\
117	0.00650423232804379\\
117.75	0.00785303315869148\\
118.5	0.00880533132719444\\
119.25	0.00932142187264726\\
120	0.00937351852798264\\
120.75	0.00896894691883869\\
121.5	0.00814796060353159\\
122.25	0.00697754194154283\\
123	0.00593889865698565\\
123.75	0.00640557092848264\\
124.5	0.00658805390340318\\
125.25	0.00649198731592775\\
126	0.00614812649046015\\
126.75	0.00558710096907493\\
127.5	0.0048419516596005\\
128.25	0.00454900041188221\\
129	0.00525309777428288\\
129.75	0.00563968722458059\\
130.5	0.00568731568779749\\
131.25	0.00541573934285767\\
132	0.00488129060469383\\
132.75	0.00412573551701235\\
133.5	0.00366697697411301\\
134.25	0.00431999179126934\\
135	0.004762162176389\\
135.75	0.00500022411994768\\
136.5	0.00503505939899233\\
137.25	0.00487200480787187\\
138	0.00451067872288495\\
138.75	0.0039502741820763\\
139.5	0.00404654422646079\\
140.25	0.00433043416333644\\
141	0.00440662436019817\\
141.75	0.00429571423479498\\
142.5	0.00403697015107032\\
143.25	0.00365251480955572\\
144	0.00316856995795561\\
144.75	0.00328198052702449\\
145.5	0.00391879236148559\\
146.25	0.00441547563072026\\
147	0.00475103798180392\\
147.75	0.00490299714570629\\
148.5	0.00485454809938103\\
149.25	0.00460744753920739\\
150	0.00417774114888593\\
};

\end{axis}

\begin{axis}[%
width=1.227\fwidth,
height=1.227\fheight,
at={(-0.16\fwidth,-0.135\fheight)},
scale only axis,
xmin=0,
xmax=1,
ymin=0,
ymax=1,
axis line style={draw=none},
ticks=none,
axis x line*=bottom,
axis y line*=left,
legend style={legend cell align=left, align=left, draw=white!15!black}
]
\end{axis}
\end{tikzpicture}%

%% file: conclusions.tex
\section{Conclusions}
In this paper, we have studied a model order reduction problem for structured nonlinear systems. To that aim, we have defined generalized transfer functions for structured systems based on the associated Volterra series. We have then shown how to construct a reduced-order model such that its generalized transfer functions interpolate those of the original model at the pre-defined frequency points. Subsequently, we have proposed an algorithm that allows determining the dominant subspaces to construct minimal interpolatory reduced-order models. Moreover, we have discussed extensions of those results to parametric nonlinear structured systems, as well as special structured systems such as  input-delay systems. Finally, we have illustrated the performance of the proposed methodology based on a couple of examples and compared it with the existing methodologies. In future work, we will focus on choosing interpolation points adaptively, which is essential to reduce computational efforts and to capture all the important dynamics of the system. To achieve this goal, one may tailor the ideas presented in \cite{morCheFB20}.

%% file: fundsol.tex
\section{Fundamental solution using frequency domain methods}\label{appen:FunSol}
The fundamental solution of a linear operator can be defined in different ways. In this work, we follow the approaches that use the frequency domain representation of the linear operator. We refer the reader to  \cite{bellman1963differential,hale2013introduction} for more details. 

Let us start by defining the unilateral Laplace transform. Given a function $\bg(\cdot)$, its unilateral Laplace transform is given by
\[  \mathscr{L}(\bg(\cdot)) = \bG(s) = \int_0^{\infty} e^{-st} \bg(t)dt.  \]
where $\bG(\cdot)$ corresponds to the frequency domain representation of $\bg(\cdot)$.

Now, let us consider a linear operator $\cL$ such as shown  in Table \ref{tab:LinearStructureExamples}. By means of the Laplace transform, we obtain the frequency domain representation of $(\cL\bx)(t)$ as follows:
\begin{equation}\label{eq:LinOptTF}
\mathscr{L}((\cL \bx)(t)) = \cK(s) \bX(s), 
\end{equation}
where $\bX(s)$ corresponds to the Laplace transform of $\bx(t)$ and $ \cK(s)$ corresponds to the frequency domain representation of the operator $\cL$. Let us assume that the inverse of the Laplace transform of $\cK^{-1}(s)$ exists and is given as
\[  \mathscr{L}^{-1}(\cK^{-1}(\cdot)) := \mphi(\cdot).   \]
Then, $\mphi(t)$ is the \emph{fundamental solution} associated to the linear operator  $\cL$. Indeed, $\mphi(t)$ is  the solution of the functional differential equation 
\[ \left(\cL \bx\right)(t)	= \delta(t), \]
where $\delta(t)$ is the Dirac delta distribution and the initial conditions are all zero. Moreover, the inhomogeneous  equation
\[ \left(\cL \bx\right)(t) = \bg(t), \]
with $\bg(\cdot)$ being a suitable function, has a solution in convolution form
\[\bx(t) = \int_0^{t}  \mphi(\sigma)\bg(t-\sigma)d\sigma. \]


%% file: tangetialinterp.tex
\section{Tangential interpolation-based MOR for MIMO systems}\label{appen:TangInterp} 
Here, we discuss a construction of an interpolating ROM for MIMO polynomial systems. Similar to the SISO case, the leading generalized transfer functions for a MIMO polynomial system are given
as follows:
\begin{subequations}\label{eq:general_TF_MIMO}
	\begin{align}
		\bF_{\mathrm L}(s_1) & =  \bC\cK^{-1}(s_1)\bB,
		\\
		\bF^{(\xi)}_{\mathrm H}(s_1,\ldots,s_{\xi+1}) &=  \bC\cK^{-1}(s_{\xi+1}) \bH_\xi \left(\cK^{-1}(s_\xi)\bB \otimes \cdots \otimes \cK^{-1}(s_{1})\bB\right),
		\\
		\bF_{\mathrm N}^{(\eta)}(s_1,\ldots, s_{\eta+1}) & =  \bC \cK^{-1}(s_{\eta+1}) \bN_\eta\left(\bI_m\otimes\cK^{-1}(s_{\eta})\bB \otimes \cdots \otimes \cK^{-1}(s_{1})\bB \right).
	\end{align}
\end{subequations}

\begin{lemma}\label{lemma:singleInterpolation}
	Consider the original system as given in \eqref{eq:struct_nonlin_sys}. Let $\sigma_i \in \C$, $i \in \{1,\ldots,\tilde r\}$, be interpolation points such that $\cK(s)$ is invertible for all $s \in \{\sigma_1,\ldots, \sigma_{\tilde r}\}$, and $\bb_i \in \C^m$ and $\bc_i\in\C^q$ for $i \in \{1,\ldots,{\tilde r}\}$ be right and left tangential directions corresponding to $\sigma_i$, respectively. Let $\bV$ and $\bW$ be defined as follows: 
	\begin{subequations}\allowdisplaybreaks
		\begin{align*}
			\cV_{\mathrm L}	&= 	\bigcup_{i= 1}^{\tr} \range{\cK^{-1}(\sigma_i)\bB\bb_{i}},
			\\
			\cV_{\mathrm N} 	&= 	\bigcup_{\eta= 1}^{d-1}\bigcup_{i= 1}^{\tr} \range{ \cK^{-1}(\sigma_i) \bN_\eta\left(\bI_m\otimes \cK^{-1}(\sigma_i)\bB\bb_i \otimes \cdots \otimes \cK^{-1}(\sigma_i)\bB\bb_i \right)},  
			\\
			\cV_{\mathrm H} 	&=  	\bigcup_{\xi= 2}^d\bigcup_{i= 1}^{\tr}  \range{ \cK^{-1}(\sigma_i) \bH_\xi \left(\cK^{-1}(\sigma_i)\bB\bb_i \otimes \cdots \otimes \cK^{-1}(\sigma_i)\bB\bb_i \right) },
			\\ 
			\cW_{\mathrm L}	&= 	\bigcup_{i= 1}^{\tr}\range{\cK^{-\top}(\sigma_i)\bC^\top\bc_i},
			\\
			\cW_{\mathrm N} 	&= 	\bigcup_{\eta= 1}^{d-1} \bigcup_{i= 1}^{\tr}\range{ \cK^{-1}(\sigma_{i}) \left(\bN_\eta \right)_{(2)}\left(\bI_m\otimes\cK^{-1}(\sigma_i)\bB\bb_{i} \otimes \cdots \otimes \cK^{-1}(\sigma_i)\bB\bb_{i} \otimes\cK^{-\top}(\sigma_i)\bC^\top\bc_i \right)}, 
			\\
			\cW_{\mathrm H} 	&=  	\bigcup_{\xi= 2}^d \bigcup_{i= 1}^{\tr} \range{ \cK^{-1}(\sigma_{i}) \left(\bH_\xi\right)_{(2)} \left(\cK^{-1}(\sigma_{i})\bB\bb_{i} \otimes \cdots \otimes \cK^{-1}(\sigma_{i})\bB\bb_{i}\otimes \cK^{-\top}(\mu_i)\bC^\top\bc_i \right) },
			\\
			\range{\bV} 	&= 	\cV_{\mathrm L}+ \cV_{\mathrm N}+ \cV_{\mathrm H},
			\\
			\range{\bW}	&= 	\cW_{\mathrm L}+ \cW_{\mathrm N}+ \cW_{\mathrm H}.
		\end{align*}
	\end{subequations}
	If a ROM is computed as shown in \eqref{eq:RedMatrices} using the projection matrices $\bV$ and $\bW$, where we assume $\bV$ and $\bW$ to be of full rank, then the following interpolation conditions are fulfilled:
	\begin{subequations}\allowdisplaybreaks
		\begin{align}
			\bF_{\mathrm L}(\sigma_i) \bb_i	&= 	\hat{\bF}_{\mathrm L}(\sigma_i)\bb_i, \label{eq:1}
			\\
			\bc_i^\top \bF_{\mathrm L}(\sigma_i)	&= 	\bc_i^\top\hat{\bF}_{\mathrm L}(\sigma_i),\label{eq:2} 
			\\
			\dfrac{d}{ds_1} \bc_i^\top \bF_{\mathrm L}(\sigma_i)\bb_i
			&= 	\dfrac{d}{ds_1} \bc_i^\top\hat{\bF}_{\mathrm L}(\sigma_i)\bb_i,  \label{eq:3}
			\\
			\bF_{\mathrm N}^{(\eta)}(\sigma_i,\ldots,\sigma_i)\left(\bI_m\otimes \bb_i^{\circled{\tiny {$\eta$}}}\right) 
			&= 	\hat \bF_{\mathrm N}^{(\eta)}(\sigma_i,\ldots,\sigma_i) \left(\bI_m\otimes \bb_i^{\circled{\tiny {$\eta$}}}\right), \label{eq:4}
			\\
			\bc_i^\top \bF_{\mathrm N}^{(\eta)}(\sigma_i,\ldots,\sigma_i)\left( \bI_m^{\circled{\tiny {$2$}}} \otimes  \bb_i^{\circled{\tiny {$\eta{-}1$}}}\right) 
			&= 	\bc_i^\top\hat{\bF}_{\mathrm N}^{(\eta)}(\sigma_i,\ldots,\sigma_i)\left( \bI_m^{\circled{\tiny {$2$}}} \otimes  \bb_i^{\circled{\tiny {$\eta{-}1$}}}\right)\label{eq:5}
			\\
			\dfrac{\partial}{\partial s_j}\bc_i^\top \bF_{\mathrm N}^{(\eta)}(\sigma_i,\ldots,\sigma_i)\left(\bI_m\otimes \bb_i^{\circled{\tiny {$\eta$}}}\right) 
			&= 	\dfrac{\partial}{\partial s_j}  \bc_i^\top \hat \bF_{\mathrm N}^{(\eta)}(\sigma_i,\ldots,\sigma_i) \left(\bI_m\otimes \bb_i^{\circled{\tiny {$\eta$}}}\right), \label{eq:6}
			\\
			\bF_{\mathrm H}^{(\xi)}(\sigma_i,\ldots,\sigma_i) \bb_i^{\circled{\tiny {$\xi$}}} 
			&= 	\hat{\bF}_{\mathrm H}^{(\xi)}(\sigma_i,\ldots,\sigma_i)  \bb_i^{\circled{\tiny {$\xi$}}}, \label{eq:7}
			\\
			\bc_i^\top \bF_{\mathrm H}^{(\xi)}(\sigma_i,\ldots,\sigma_i)\left( \bI_m \otimes  b_i^{\circled{\tiny {$\xi{-}1$}}}\right) 
			&= 	\bc_i^\top\hat{\bF}_{\mathrm H}^{(\xi)}(\sigma_i,\ldots,\sigma_i)\left( \bI_m \otimes  \bb_i^{\circled{\tiny {$\xi{-}1$}}}\right),\label{eq:8}
			\\
			\dfrac{\partial}{\partial s_j}\bc_i^\top \bF_{\mathrm H}^{(\xi)}(\sigma_i,\ldots,\sigma_i)\bb_i^{\circled{\tiny {$\xi$}}} 
			&= 	\dfrac{\partial}{\partial s_j}  \bc_i^\top\hat{\bF}_{\mathrm H}^{(\xi)}(\sigma_i,\ldots,\sigma_i) \bb_i^{\circled{\tiny {$\xi$}}}\label{eq:9}
		\end{align}
	\end{subequations}
	where $i \in \{1,\ldots, \tilde r\}$, $\xi \in \{2,\ldots,d\} $, $\eta \in \{1,\ldots,d\}$ and $\dfrac{\partial}{\partial s_j}$ denotes the partial derivative with respect to $s_j$ of a given function. 
\end{lemma}
\begin{proof}
	The proof of these interpolation conditions  follows exactly the one of \Cref{thm:gen_interpolation}. 
\end{proof}